\documentclass[12pt]{article}

\usepackage[left=30mm, right=30mm, top=30mm, bottom=30mm, marginparwidth=20mm]{geometry}
\usepackage{setspace}
\usepackage{amsmath}
\usepackage{amsthm}
\usepackage{amssymb}
\usepackage{amsfonts}
\usepackage{mathtools}
\usepackage[capitalize]{cleveref}
\usepackage[style=alphabetic]{biblatex}
\usepackage{tikz-cd}
\usepackage{mathrsfs}
\usepackage{todonotes}
\usepackage{bbm}
\usepackage{quiver}
\usepackage{xurl}

\newtheorem{definition}{Definition}[section]
\newtheorem{lemma}[definition]{Lemma}
\newtheorem{theorem}[definition]{Theorem}
\newtheorem{proposition}[definition]{Proposition}
\newtheorem{corollary}[definition]{Corollary}

\newtheorem*{theorem*}{Theorem}
\newtheorem*{corollary*}{Corollary}
\newtheorem*{lemma*}{Lemma}

\newcommand{\stwotwo}[4]{\left(\begin{smallmatrix}
    #1 & #2\\
    #3 & #4
\end{smallmatrix}\right)}

\newcommand{\twotwo}[4]{\begin{pmatrix} 
    #1 & #2 \\
    #3 & #4 
  \end{pmatrix}}

\DeclareMathOperator{\GL}{GL}
\DeclareMathOperator{\gldim}{gl \ dim}
\DeclareMathOperator{\jacrad}{rad}
\DeclareMathOperator{\rad}{rad}
\DeclareMathOperator{\Hom}{Hom}
\DeclareMathOperator{\End}{End}
\DeclareMathOperator{\diag}{diag}
\DeclareMathOperator{\row}{row}
\DeclareMathOperator{\Mod}{Mod}
\DeclareMathOperator{\charstc}{char}
\DeclareMathOperator{\Tr}{Tr}
\DeclareMathOperator{\fg}{fg}

\DeclareMathOperator{\dgend}{dg-End}
\DeclareMathOperator{\dghom}{dg-Hom}
\DeclareMathOperator{\per}{per}

\newcommand{\mb}[1]{\mathbb{#1}}
\newcommand{\ms}[1]{\mathscr{#1}}

\newcommand{\modcat}{{\operatorname{-Mod}}}

\newcommand{\funct}[1]{\mathrm{#1}}
\newcommand{\ind}{\funct{ind}}
\newcommand{\infl}{\funct{infl}}
\newcommand{\res}{\funct{res}}

\newcommand{\parind}{\funct{i}}
\newcommand{\phorind}{\funct{I}}

\newcommand{\coeffs}{R}
\newcommand{\triv}{\mathbbm{1}}
\newcommand{\indicate}{1}

\newcommand{\padicfield}{F}
\newcommand{\resfield}{k}
\newcommand{\resinv}{\resfield^{\times}}
\newcommand{\roi}{\mathcal{O}_{\padicfield}}

\newcommand{\finite}[1]{#1_f}
\newcommand{\oppos}[1]{\bar{#1}}

\newcommand{\alggp}[1]{\mathrm{#1}}
\newcommand{\padicgp}{G}
\newcommand{\redgp}{\alggp{\padicgp}}
\newcommand{\gln}{\alggp{GL}_{n}}
\newcommand{\glnp}{\gln(\padicfield)}
\newcommand{\fingp}{\finite{G}}

\newcommand{\iwa}{I}

\newcommand{\maxcomp}{K}
\newcommand{\prop}[1]{#1^1}
\newcommand{\propiwa}{\prop{\iwa}}
\newcommand{\propmax}{\prop{\maxcomp}}
\newcommand{\weyl}{W}
\newcommand{\torus}{T}

\newcommand{\roots}[1]{\mathtt{#1}}
\newcommand{\rootset}{\roots{P}}
\newcommand{\finroots}{\finite{\roots{S}}}

\newcommand{\finbor}{\finite{\iwa}}
\newcommand{\finuni}{\finite{\propiwa}}

\newcommand{\finuniopp}{\oppos{\finuni}}
\newcommand{\fintor}{\finite{\torus}}
\newcommand{\finweyl}{\finite{\weyl}}

\newcommand{\iwaupper}{\iwa_+}
\newcommand{\iwadiag}{\iwa_0}
\newcommand{\iwalower}{\iwa_-}
\newcommand{\iwaoppupp}{\oppos{\iwaupper}}

\newcommand{\hecke}{\mathsf{H}}
\newcommand{\hecgln}[1]{\hecke_{#1}(n)}
\newcommand{\hecglnco}{\hecgln{\coeffs}}
\newcommand{\hecglnq}{\hecgln{q}}
\newcommand{\schur}{\mathsf{S}}
\newcommand{\schgln}[1]{\schur_{#1}(n)}
\newcommand{\schglnco}{\schgln{\coeffs}}
\newcommand{\schglnq}{\schgln{q}}
\newcommand{\bigschur}{\schur_q(N,n)}
\newcommand{\alt}[1]{\hat{#1}}
\newcommand{\altschur}{\alt{\schur}_q(N,n)}

\newcommand{\smrep}{\Mod(\padicgp)}
\newcommand{\hecglob}{\hecke(\padicgp)}
\newcommand{\block}{\mathcal{B}}
\newcommand{\uniblock}{\block_{1}(\padicgp)}
\newcommand{\unihec}{\hecke_{1}(\padicgp)}
\newcommand{\annihilator}{\mathcal{I}}
\newcommand{\anniblock}{\block_{1}'(\padicgp)}
\newcommand{\annihec}{\hecke_{1}'(\padicgp)}

\newcommand{\heccomp}{\hecke(\maxcomp)}

\newcommand{\finrep}{\Mod(\fingp)}
\newcommand{\hecfin}{\hecke(\fingp)}
\newcommand{\uniblockfin}{\block_{1}(\fingp)}
\newcommand{\hecfinuni}{\hecke_{1}(\fingp)}
\newcommand{\finanni}{\finite{\annihilator}}

\newcommand{\annigen}{Q}
\newcommand{\ourgen}{V}
\newcommand{\unigen}{\Gamma}
\newcommand{\project}{P}

\newcommand{\funig}{\finite{\unigen}}
\newcommand{\fourg}{\finite{\ourgen}}
\newcommand{\fannig}{\finite{\annigen}}
\newcommand{\finproj}{\finite{\project}}

\newcommand{\derivedcat}{D}
\newcommand{\homtpycat}{K}
\newcommand{\bounder}{{\derivedcat}^b}
\newcommand{\dermod}[1]{\derivedcat(#1)}
\newcommand{\htpymod}[1]{\homtpycat(#1)}
\newcommand{\bdermod}[1]{\bounder(#1_{\fg})}

\newcommand{\cx}[1]{#1^{\bullet}}
\newcommand{\gencx}{\cx{\ourgen}}
\newcommand{\homlgy}{H}
\newcommand{\homlgycx}{\cx{\homlgy}}
\newcommand{\nhomlgy}{\homlgy^n}

\title{The Derived Unipotent Block of $p$-Adic $\GL_2$ as Perfect Complexes over a dg Schur Algebra}
\author{Rose Berry}

\begin{document}

\maketitle

\begin{abstract}
    For a $p$-adic field $F$ of residual cardinality $q$, we provide a triangulated equivalence between the bounded derived category $\bdermod{\uniblock}$ of finitely generated unipotent representations of $\padicgp=\GL_2(F)$ and perfect complexes over a dg enriched Schur algebra, in the non-banal case of odd characteristic $l$ dividing $q+1$. The dg Schur algebra is the dg endomorphism algebra of a projective resolution of a direct sum $V$ of the parahoric inductions of the trivial representations of the reductive quotients of $\padicgp$, and $V$ is shown to be a classical generator of $\bdermod{\uniblock}$.
\end{abstract}

\section{Introduction}

Let $\padicfield$ be a non-archimedean local field with residue field $\resfield$, of characteristic $p$ and cardinality $q$, and let $\padicgp$ be the $F$-points of a reductive algebraic group over $\padicfield$.

By \cite{bernstein1984centre}, the category of smooth complex representations of $\padicgp$ is known to decompose into blocks, which are parameterised by inertial supercuspidal support. Furthermore, in almost all cases the structure of these blocks is known via the theory of types(\cite{adler2024structure, adler2024reduction}). Write $\uniblock$ for the unipotent block, that is, the block containing the trivial representation. For $\padicgp=\glnp$, it was shown, first in \cite{bushnell1999types} and expanded on in \cite{dat2017functoriality}, that each block of $\padicgp$ is equivalent to some $\block_1(H)$, where $H$ is a finite product of general linear groups over finite extensions of $\padicfield$. Furthermore, $\uniblock$ is equivalent to the category of modules over the Iwahori-Hecke algebra $\hecglnco$, which has an explicit description in terms of generators and relations, and whose modules are well-understood.

A natural question to ask is how much of this structure passes to the category of smooth $R$-representations for $R$ an algebraically closed field of characteristic $l$, which we take to be different from $p$. In the case of $l$ banal, that is, not dividing the order of the reductive quotient of any parahoric subgroup, all the same results hold. However, in non-banal characteristic much less is known.

It is shown in \cite{vigneras1998induced} that, for $\padicgp=\gln(\padicfield)$, there is the same decomposition into blocks (this is however known to fail for other choices of $\padicgp$ - see \cite{dat2018simple}). Furthermore, there is the same reduction of blocks to the case of $\uniblock$ for some product of $\gln(\padicfield)$ (see \cite{dat2018equivalences}, \cite{chinello2018blocks}). However, the equivalence of $\uniblock$ with the category of modules over $\hecglnco$ fails. Despite this, some partial results are still known:

Vign\'eras defines in \cite{vigneras2003schur} a subcategory $\anniblock$ of $\uniblock$ given by the representations annihilated by a certain ideal $\annihilator$ of the global Hecke algebra $\hecglob$ of $\padicgp$. She shows that $\anniblock$ is equivalent to the category of modules over an algebra closely related to $\hecglnco$, the Schur algebra $\schglnco$. Furthermore, some power of $\annihilator$ annihilates $\uniblock$. Vign\'eras achieves this by giving an explicit progenerator $\annigen$ of the subcategory whose endomorphisms give $\schglnco$, but also observes that $\schglnco$ is the endomorphisms of a much simpler representation $\ourgen$, whose annihilator is $\annihilator$.

We seek to improve this state of affairs at the derived level. Write $\bdermod{\uniblock}$ for the bounded derived category of finitely generated unipotent representations, and for a triangulated category $\mathcal{T}$ with object $G$ write $\langle G\rangle_{\mathcal{T}}$ for the subcategory classically generated by $G$. We obtain the following main theorem:

\begin{theorem*}[Theorems 6.8 and 6.10]
    Suppose that $n=2$ and $l$ odd dividing $q+1$. Then $\bdermod{\uniblock}=\langle \annigen\rangle_{\bdermod{\uniblock}}=\langle \ourgen\rangle_{\bdermod{\uniblock}}$.
\end{theorem*}

We may immediately improve this via the theory of dg algebras. Let $\gencx$ be a projective resolution of $\ourgen$ in $\smrep$, and write $\dgend$ for the dg endomorphism algebra of a complex and $\per$ for the perfect complexes over a dg algebra.

\begin{corollary*}[Corollary 6.11]
    Under the same hypotheses as the theorem, there is a triangulated equivalence $\bdermod{\uniblock}\simeq\per(\dgend(\gencx))$.
\end{corollary*}

Due to the relative simplicity of $\ourgen$, its projective resolution $\gencx$ can be described explicitly, and future work of the author seeks to give an explicit description of its dg endomorphisms.

To establish that $\annigen$ classically generates $\bdermod{\uniblock}$, we use the results of \cite{dat2009finitude} to show that the categories under consideration are Noetherian, and extend a result of \cite{deng2016affine} to positive characteristic to show the following:

\begin{lemma*}[Theorem 3.15]
    For any $n$ and any $l\neq p$, the Schur algebra $\schglnco$ has finite global dimension.
\end{lemma*}

To show an equivalence between the categories generated by $\annigen$ and $\ourgen$ in the case $n=2$ and $l$ odd dividing $q+1$, we use that the unipotent block $\uniblockfin$ of the finite reductive quotient $\fingp=\gln(\resfield)$ has cyclic defect group in this characteristic, and that it contains certain finite analogues $\fannig$ and $\fourg$ of $\annigen$ and $\ourgen$. This allows us to use results of \cite{ackermann2006loewy} to describe explicitly the full structure of simple and projective modules of $\uniblockfin$, and to describe $\fannig$ and $\fourg$ in terms of these. Thus we may show the following:

\begin{lemma*}[Lemma 4.13]
    $\langle \fannig\rangle_{\bdermod{\uniblockfin}}=\langle \fourg\rangle_{\bdermod{\uniblockfin}}$.
\end{lemma*}

We then attempt to lift this equivalence into the $p$-adic setting. However, a barrier occurs, as $\annigen$ is not parahorically induced from $\fannig$. Resolving this amounts to showing the following property of the annihilator $\finanni$ of $\fourg$, which is a 2-sided ideal of the finite group algebra and hence may be viewed as a representation of $\fingp$:

\begin{lemma*}[Lemma 5.4]
    $\phorind_{\fingp,\maxcomp}^{\padicgp}\finanni\subseteq\annihilator$.
\end{lemma*}

We show this by giving a full and explicit description of $\finanni$ and showing it annihilates $\ourgen$. To get this explicit description, we again use the structure of $\uniblockfin$ in terms of projectives, together with results from \cite{paige2014projective} giving explicit formulae for their endomorphisms. Thus the final proof rests on a long explicit calculation in the global Hecke algebra of the $p$-adic group.

The main theorem and its corollary allow us to describe $\bdermod{\uniblock}$ in terms of a dg algebra whose zeroth cohomology is $\schglnco$. The current work of the author is in describing this dg algebra, up to quasi-isomorphism, explicitly, to enable an explicit description of its perfect complexes and thus of $\bdermod{\uniblock}$.

The methods of this paper, specifically the second and third lemmas, rely on explicit knowledge of the structure of $\uniblockfin$. This structure is known in more generality than just $n=2$ and $l$ odd dividing $q+1$: it is known whenever $\uniblockfin$ has cyclic defect group. We thus expect this method to generalise to this broader case.

Furthermore, the Schur algebra can be defined for any $\padicgp$, and it seems likely that the proof of the first lemma could extend to this case. The generators $\ourgen$, $\annigen$, $\fourg$, and $\fannig$ can all be defined in this generality, and one might hope that the other two lemmas hold in this case too, though the method in this paper is unlikely to extend. Should a more general method be found, the main theorem could then be extended to other cases where the block decomposition is known to hold.

\subsection{Organisation of the Paper}

\cref{preliminaries} establishes our notation and recalls the relevant results of \cite{vigneras2003schur} and \cite{dat2009finitude} that we will base our results on. In \cref{schur algebra structure} we will recall the notion of an affine cellular algebra, and extend the work of \cite{deng2016affine} to give an affine cellular structure on $\schglnco$ that satisfies additional idempotence properties, which imply that $\schglnco$ has finite global dimension.

\cref{finite block} recalls from \cite{ackermann2006loewy} the structure of $\uniblockfin$ in our choice of characteristic, and uses this to establish the structure of $\finanni$ and the equivalence between $\fourg$ and $\fannig$. Then in \cref{p-adic lift} we lift this equivalence to one between $\ourgen$ and $\annigen$. Finally, in \cref{derived categories} we first recall some general properties of dg algebras and their derived categories, and then use these to establish our equivalence.

\subsection{Acknowledgements}

I would like to thank the Engineering and Physical Sciences Research Council, grant T00046, for funding my PhD studentship, which has enabled all of this research. I would also like to thank my PhD supervisors, Professors Shaun Stevens and Vanessa Miemietz, for their constant guidance and support.

\section{Preliminaries}\label{preliminaries}

Let $\padicfield$ denote a non-archimedean local field, with ring of integers $\roi$, uniformiser $\varpi$, and residue field $\resfield$. Write $q=\left|\resfield\right|$ and $p=\charstc(\resfield)$. Let $\redgp=\gln$, and write $\padicgp=\redgp(\padicfield)$ and $\fingp=\redgp(\resfield)$.

Fix in $\fingp$ a minimal parabolic subgroup $\finbor$; we may without loss of generality take $\finbor$ to be the the upper triangular matrices. Similarly, we have the unipotent radical of $\finbor$, written $\finuni$, which is the unipotent upper triangular matrices. Fix also a maximal split torus $\fintor$ in $\finbor$, which we take to be the diagonal matrices. We shall also need $\finuniopp$, the transpose of $\finuni$, that is, the unipotent lower triangular matrices.

Fix in $\padicgp$ an Iwahori subgroup $\iwa$, which we take to be the preimage of $\finbor$ under the quotient $\roi\rightarrow\resfield$. Similarly, we have the pro-$p$ radical of the Iwahori subgroup, written $\propiwa$, which is the preimage of $\finuni$, and a maximal compact subgroup $\maxcomp$ with pro-$p$ radical $\propmax$, which are the preimages of $\fingp$ and $1$ respectively.

Let $\weyl$ denote the extended affine Weyl group of $\padicgp$, and $\finweyl$ the finite Weyl group of $\padicgp$, with $\finroots$ its simple reflections. We can view these as subgroups of $\padicgp$, with $\finweyl$ being the permutation matrices, $\finroots$ the simple transpositions, and $\weyl$ being the product of these with diagonal matrices with entries powers of $\varpi$.

Let $\coeffs$ be an algebraically closed field of characteristic $l\neq p$. We write $\hecfin$ for the group algebra of $\fingp$ over $\coeffs$, and $\finrep$ for the category of $\fingp$-representations over $\coeffs$, that is, modules over $\hecfin$. Similarly, write $\smrep$ for the smooth (left)-representations of $G$ over $R$. We fix an $R$-valued Haar measure on $\padicgp$ with $\mu(\propmax)=1$. Then $\smrep$ is isomorphic to the category of nondegenerate modules over the global Hecke algebra $\hecglob$ associated to $\mu$.

The unipotent block of $\uniblockfin$ is the block containing the trivial representation $\triv$. Denote the corresponding direct summand of $\hecfin$ by $\hecfinuni$ and its direct summand of $\finrep$ by $\uniblockfin$. Similarly, the the block of $\smrep$ containing the trivial representation $\triv$ is called the unipotent block and written $\uniblock$. The corresponding block in $\hecglob$ is written $\unihec$.

We write $\ind_H^{G}$ for compact induction from $H$ to $G$, as well as $\infl_{H}^G$ for inflation from $H$ to $G$, and $\res_H^G$ for restriction from $G$ to $H$. For a parabolic subgroup $C$ with levi subgroup $L$, write $\parind_{L,C}^G$ for (unnormalised) parabolic induction from $L$ to $G$ via $C$. Similarly, for a parahoric subgroup $J$ with reductive quotient $M$, write $\phorind_{M,J}^G=\ind_{J}^G\infl_{M}^J$ for parahoric induction from $M$ to $G$ via $J$.

Let $\finproj=\ind_{\finbor}^{\fingp}\triv$, and let $\project=\phorind_{\fingp,\maxcomp}^{\padicgp}\finproj=\ind_{I}^G\triv$. Recall that the Iwahori-Hecke algebra of $\padicgp$ with respect to $\iwa$, written $\hecglnco$, is $\End(P)$, or equivalently by Frobenius reciprocity the left-and-right-$\iwa$-invariant functions $\padicgp\rightarrow R$ with multiplication given by convolution. This has a presentation, called the Bernstein presentation (see \cite{vigneras2006algebres}), which is defined over $\ms{Z}=\mb{Z}[q^{\pm\frac{1}{2}}]$. Thus, using these generators and relations, we may define the Iwahori-Hecke algebra over $\ms{Z}$, written $\hecglnq$. We can recover $\hecglnco$ as $\coeffs\otimes_{\ms{Z}}\hecglnq$.

For each subset $\rootset$ of $\finroots$, let $\weyl_{\rootset}$ denote the subgroup of $\finweyl$ generated by $\rootset$. Then write $x_{\rootset}$ for the element $\indicate_{\iwa\weyl_{\rootset}\iwa}\in\hecglnq$.

\begin{definition}\label{schur algebra definition}
    The \emph{Schur algebra} of $\padicgp$ over $\coeffs$ is the $\coeffs$-algebra $\schglnco$ defined by

\begin{align*}
    \schglnco=\End_{\hecglnco}\left\{\bigoplus_{{\rootset}\subseteq\finroots}x_{\rootset}\hecglnco\right\}.
\end{align*}
\end{definition}

We can also define the Schur algebra $\schglnq$ over $\ms{Z}$ as for the Hecke algebra, and we again have $\schglnco=\coeffs\otimes_{\ms{Z}}\schglnq$.

We can describe subsets of $\finroots$ explicitly. Recall that a \emph{composition} of $n$ is a tuple $\lambda=\left\{\lambda_1,\dots,\lambda_M\right\}$ of positive integers that sums to $n$. Define ${\rootset}(\lambda)$ to be the subset of $\finroots$ of all $(i \ i+1)$ with $\sum_{m'=1}^{m}\lambda_{m'} \leq i < \sum_{m'=1}^{m+1}\lambda_{m'}$ for some $m$. This gives a bijection between compositions of $n$ and subsets of $\finroots$.

Let $\annihilator$ be the annihilator in $\hecglob$ of the module $\project$, which lives in the unipotent block. Let $\anniblock$ be the category of $\hecglob$-modules annihilated by $\annihilator$, that is, the category of modules over $\annihec=\hecglob/\annihilator$. This quotient factors through $\unihec$, that is, $\anniblock$ is a subcategory of $\uniblock$. In characteristic zero, or more generally when the characteristic of $\coeffs$ doesn't divide the order of $\fingp$, we have that $\anniblock$ and $\uniblock$ coincide, and moreover are equivalent to the category of modules over the Iwahori-Hecke algebra $\hecglnco$. The situation in general is more complex:

\begin{theorem}\label{Vigneras background}
    Some finite power $\annihilator^N$ of $\annihilator$ annihilates $\uniblock$. Furthermore, $\anniblock$ is equivalent to the category of modules over the Schur algebra $\schglnco$.
\end{theorem}

\begin{proof}
    This is the `Main Theorem' from Section $2$ of the introduction of \cite{vigneras2003schur}.
\end{proof}

This is proven via explicit construction of a progenerator of $\anniblock$, whose endomorphism algebra is $\schglnco$. In attempting to find derived extensions of this equivalence, we shall make use of the explicit structure of this progenerator.

For each standard parahoric subgroup $\iwa\subseteq J\subseteq \maxcomp$, let $M_J=J/J^1$ denote its reductive quotient, let $U_J$ denote the unipotent radical of a minimal parabolic subgroup of $M_J$, and let $\chi_J$ denote a generic character of $U_J$. We define the finite Gelfand Graev module of $J$ to be $\Gamma_J=\ind_{U_J}^{M_J}\chi_J$. Note that this is independent of choice of $U_J$ and $\chi_J$. In what follows, we shall simply refer to $\Gamma_K$ as \emph{the} Gelfand Graev module. Write $J_f=J/K^1$, let $\funig=\bigoplus_J\parind_{M_J,J_f}^{\fingp}\Gamma_J$, and put $\unigen=\phorind_{\fingp,\maxcomp}^{\padicgp}\funig$.

\begin{theorem}\label{Vigneras generator}
    $\anniblock$ has progenerator $\annigen=\unigen/\annihilator\unigen$.
\end{theorem}

\begin{proof}
    \cite{vigneras2003schur}, Proposition D10, says that $\annigen$ is a progenerator. The definition of $\annigen$ is their equation 8, and the definition of $\unigen$ is their Proposition D5(1), where it is important to note that we have removed repeated copies of the same projective summands, and we have specialised her definition for a general reductive group to our particular $\padicgp$.
\end{proof}

We shall also need finiteness of the various categories we consider.

\begin{theorem}\label{everything is noetherian}
    $\smrep$ is noetherian.
\end{theorem}

\begin{proof}
    This is \cite{dat2009finitude}, Theorems 1.3 and 1.5.
\end{proof}

Thus $\hecglob$ is Noetherian, $\annihilator$ is finitely generated, and $\unihec$ and $\annihec$ are Noetherian.

\section{The Structure of the Schur Algebra}\label{schur algebra structure}

We shall need a key property of the Schur algebra to proceed: that it has finite global dimension. We shall obtain this by finding a suitable affine cellular structure. We recall here the definition. Let $A$ be an $R$-algebra with anti-involution $i$.

\begin{definition}\label{affine cell ideal}
    A 2-sided ideal $J$ in $A$ such that $i(J)=J$ is called an \emph{affine cell ideal} if there are
    \begin{itemize}
        \item a free $\coeffs$-module of finite rank $V$,
        \item a finitely generated commutative $\coeffs$-algebra $B$ with involution $\sigma$,
        \item and a left $A$-module structure on $\Delta=V\otimes_{\coeffs} B$ that is compatible with the regular right $B$-module structure,
    \end{itemize}
    such that, if we define a right $A$-module structure on $\Delta'=B\otimes_{\coeffs} V$ by $xa=\tau^{-1} (i(a)\tau(x))$ where $\tau:\Delta'\rightarrow\Delta$, $b\otimes v\mapsto v\otimes b$, there is an isomorphism of $A$-$A$-bimodules $\alpha: J\rightarrow \Delta\otimes_B\Delta'=V\otimes_{\coeffs} B\otimes_{\coeffs} V$ making the following diagram commute:
    \[\begin{tikzcd}
        J & {V\otimes_{\coeffs} B\otimes_{\coeffs} V} \\
        J & {V\otimes_{\coeffs} B\otimes_{\coeffs} V}
        \arrow["\alpha", from=1-1, to=1-2]
        \arrow["{v\otimes b\otimes v' \mapsto v'\otimes\sigma(b)\otimes v}", from=1-2, to=2-2]
        \arrow["i"', from=1-1, to=2-1]
        \arrow["\alpha"', from=2-1, to=2-2]
    \end{tikzcd}\]
\end{definition}

\begin{definition}\label{affine cellular}
    $A$ is said to be \emph{affine cellular} if there is an $\coeffs$-module decomposition $A=\bigoplus_{k=1}^KJ_k'$ such that, for all $k$, $i(J_k')=J_k'$ and $J_k=\bigoplus_{k'=1}^kJ_{k'}'$ is a 2-sided ideal in $A$ such that $J_k'=J_k/J_{k-1}$ is an affine cell ideal in $A/J_{k-1}$.
\end{definition}

Write $V_k$ and $B_k$ for the $V$ and $B$ as in \cref{affine cell ideal} that give an affine cell ideal structure for $J_k'$.

\begin{definition}
    Using the same notation as in \cref{affine cellular}, an affine cellular algebra is said to be \emph{idempotent affine cellular} if, for all $k$, we have that $J_k'$ is generated as a 2-sided ideal in $A/J_{k-1}$ by a nonzero idempotent.
\end{definition}

We will use the following two lemmas to simplify the task of finding such a structure.

\begin{lemma}[{\cite[Theorem 4.3(1)]{koenig2012affine}}]\label{idempotents generate}
    Let $A$ be affine cellular with notation as above. If, as a 2-sided ideal in $A/J_{k-1}$, we have that $J_k'^{2}=J_k'$ and $J_k'$ contains a nonzero idempotent $e$, then $e$ generates $J_k'$ as a 2-sided ideal in $A/J_{k-1}$.
\end{lemma}

\begin{lemma}\label{idempotent preserves affine cellular}
    Suppose $e$ is an idempotent in $A$.
    
    \begin{enumerate}
        \item If $i(e)=e$, and if $A$ is affine cellular, with notation as above, then so is $eAe$, with the same $B_k$, and $J_k'$ replaced with $eJ_k'e$.
        \item If $AeA=A$, then restriction of scalars gives a Morita equivalence from $A$ to $eAe$.
        \item If $j\in eAe$ generates a 2-sided ideal $J$ in $A$, then it generates $eJe$ in $eAe$.
    \end{enumerate}
\end{lemma}

\begin{proof} 
    The first claim is \cite{yang2014affine}, Lemma 3.3, and the second is Proposition 2.4 from the same paper.

    The third claim is a quick direct calculation:

    \begin{align}
        \begin{split}
            eJe &= eAjAe \\
            &= eAejeAe
        \end{split}
    \end{align}

    where the last line follows as $j\in eAe$ and hence $eje=j$.
\end{proof}

The next result will allow us to restrict our attention to $\hecglnq$ and $\schglnq$, as the (idempotent) affine cellular structures over $\coeffs$  will then follow immediately.

\begin{lemma}[{\cite[Lemma 2.4]{koenig2012affine}}]\label{extension preserves affine cellular}
    Let $\coeffs'$ be a noetherian domain that is an $\coeffs$-algebra, $A$ an (idempotent) affine cellular $\coeffs$-algebra. Then $\coeffs'\otimes_{\coeffs} A$ is an (idempotent) affine cellular algebra with affine cellular structure induced from that on $A$ by taking the tensor product throughout by $\coeffs'$.
\end{lemma}

The point of considering this notion is that it allows us to prove finite global dimension.

\begin{theorem}[{\cite[Theorem 4.4]{koenig2012affine}}]\label{finite global dimension}
    Suppose $A$ is idempotent affine cellular, with notation as above, and suppose $\jacrad(B_k)=0$ and $\gldim(B_k)<\infty$ for all $k$. Then $\gldim(A)<\infty$.
\end{theorem}

Observe that neither of these final two conditions are preserved under extension of scalars in general. However, we will see that they are preserved for the particular $B_k$ we work with.

The literature also makes use of the following auxiliary object, a \emph{two-parameter Schur algebra}, whose affine cellular structure we shall prove as an intermediate step. Fix some $N>n$, and let an $N$-composition of $n$ be an $N$-tuple of \emph{nonnegative} integers that sum to $n$. Let the set of all such be denoted $\Lambda(N,n)$. For $\lambda\in\Lambda(N,n)$ we define ${\rootset}(\lambda)$ to be ${\rootset}(\lambda')$ where $\lambda'$ is the composition of $n$ given by deleting all the zero entries from $\lambda$. From this we may define a new Schur algebra:

\begin{align}\label{bigschur definition}
    \bigschur=\End_{\hecglnq}\left(\bigoplus_{\lambda\in\Lambda(N,n)}x_{{\rootset}(\lambda)}\hecglnq\right).
\end{align}

Let $\Lambda_0(N,n)$ denote the elements of $\Lambda(N,n)$ where all zero entries occur after every nonzero entry. There is precisely one for each composition of $n$. Thus by sending each composition to its corresponding element of $\Lambda_0(N,n)$, we may include $\schglnq$ in $\bigschur$.

In \cite{yang2014affine}, Definition 3.1, they give another construction of a two-parameter Schur algebra, which we denote $\altschur$ (in their notation it would be written $\hat{S}_v(N,n))$. Despite its very different construction, it in fact coincides with $\bigschur$:

\begin{proposition}[{\cite[Proposition 7.4]{varagnolo1999decomposition}}]\label{bigschur is altschur}
    There is an isomorphism $\bigschur\rightarrow\altschur$, which we shall denote $x\mapsto\alt{x}$.
\end{proposition}

A great deal is known about the affine cellular structure of $\altschur$, which can by the previous proposition be transferred to $\bigschur$.

\begin{proposition}\label{stuff from dengyang}
    $\altschur$ is idempotent affine cellular. We shall denote this structure using the notation of \cref{affine cellular}, but with $\alt{J}$ instead of $J$. Furthermore:
    \begin{enumerate}
        \item\label{dystuff 1} $\altschur$ has a basis $\alt{e}_A$, where $A$ are certain $\mb{Z}$-by-$\mb{Z}$ matrices.
        \item\label{dystuff 2} The involution $\alt{i}$ on $\altschur$ is given by $\alt{i}(\alt{e}_A)=\alt{e}_{A^t}$.
        \item\label{dystuff 3} For each $k$, the module $\alt{J}_k'$ contains the element $\alt{l}_k=\alt{e}_{\diag(\lambda^{(k)})}$, where $\lambda^{(k)}\in\Lambda(N,n)$ has decreasing entries, and for $\lambda\in\Lambda(N,n)$ we write $\diag(\lambda)$ for the $\mb{Z}$-by-$\mb{Z}$ matrix which is zero except in the diagonal $N$-by-$N$ blocks, which are diagonal with $m$th entry $\lambda_m$.
        \item\label{dystuff 4} $\alt{l}_k\alt{e}_A=\alt{e}_A$ for any $A=(a_{ij})$ with $\row(A)=\lambda^{(k)}$, where \\ $\row(A)=\left(\sum_{j\in\mathbb{Z}}a_{ij}\right)_{1\leq i \leq N}$.
        \item\label{dystuff 5} For each $k$, $B_k$ is of the form $\ms{Z}[X_1,\dots, X_{m_k}, X_{i_1}^{-1}, X_{i_{n_k}}^{-1}]$ for some $\left\{i_1,\dots,i_{n_k}\right\}\subseteq\left\{1,\dots, m_k\right\}$.
    \end{enumerate}
    Furthermore, all the above also holds for $\bigschur$.
\end{proposition}

\begin{proof}
    By \cite{nakajima2015affine} (see also \cite{cui2015affine}, Theorem 4.7), $\altschur$ is affine cellular. The details of the structure given in this paper are expounded in Sections 3 and 4 of \cite{deng2016affine}:

    The first claim is their Definition 3.1 (with their Equation 3.3 characterising the specific matrices $A$, and where they write $\alt{e}_A$ as just $e_A$).

    The second claim is their Equation 4.6 (they call the involution $\tau$).

    Their Proposition 4.3, meanwhile, defines certain sets $\mathfrak{c}_{\lambda}$ for every $\lambda\in\Lambda(N,n)$ with decreasing entries. Then, by their Equation 4.7 and the preceding paragraphs on page 443, the $\alt{J}_k'$ are spanned by certain $\mathfrak{c}_{\lambda^{(k)}}$. (In their notation, the modules $\alt{J}_k'$ are written as $C_i'$).

    From their Definition 3.1, for any $\lambda\in\Lambda(N,n)$, there is the element $\alt{l}_{\lambda}=\alt{e}_{\diag(\lambda)}$ (which they again just write as $l_{\lambda}$). Their Proposition 4.1 says that there is some other element, written $\{l_{\lambda}\}$, is equal to $\alt{l}_{\lambda}$. Then their Equation 4.3 says that $\{l_{\lambda}\}\in\mathfrak{c}_{\lambda}$ when the entries of $\lambda$ are decreasing. Thus we get the third claim.

    The fourth claim is Equation 3.6 (with Equation 3.4 giving the definition of $\row(A)$).

    In their Equation 4.8, they define, for every $\lambda\in\Lambda(N,n)$ with decreasing entries, rings $B_{\lambda}$=$\ms{Z}[X_1,\dots, X_{\lambda_1}, X_{i_1}^{-1}, X_{i_{N}}^{-1}]$ where $i_j=\lambda_1-\lambda_{j+1}$ and we set $\lambda_{N+1}=0$. Their Proposition 4.4 then gives a so-called generalised matrix algebra structure on the $\alt{J}_k'$, which is given over $B_{\lambda^{(k)}}$. But Proposition 2.2 of \cite{cui2016affine} (which combines Propositions 2.2 and 2.3 of \cite{koenig2012affine}) says that this is exactly an affine cellular structure, with $B_k=B_{\lambda^{(k)}}$. Thus we have the fifth claim.

    The results then hold for $\bigschur$ by applying the isomorphism of \cref{bigschur is altschur}.
\end{proof}

To reduce to $\schglnq$, we need some further properties of the idempotents.

\begin{lemma}\label{idempotents live in our algebra}
    For all $\lambda\in\Lambda(N,n)$, we have that $e_{\diag(\lambda)}$ is the identity map on $x_{{\rootset}(\lambda)}\hecglnq$. Additionally, for all $k$, $l_k$ lies in $\schglnq$, and $l_k$ is an idempotent generating the affine cell ideal $J_k'$. 
\end{lemma}

\begin{proof}
    The first claim is Equation 8.4 of the Appendix of \cite{deng2016affine}.

    Now we observe that every decreasing tuple in $\Lambda(N,n)$ is an element of $\Lambda_0(N,n)$, and that by claim (3) of \cref{stuff from dengyang} that $l_k=e_{\diag(\lambda^{(k)})}$ and that $\lambda^{(k)}$ is decreasing. Thus $l_k\in\schglnq$.

    Finally, by claim (4) of \cref{stuff from dengyang}, $l_k$ is idempotent, and as $\bigschur$ is idempotent affine cellular the $J_k'$ are idempotent, so $l_k$ generates $J_k'$ by \cref{idempotents generate}.
\end{proof}

Let $e$ denote the identity of $\schglnq$; it is an idempotent in $\bigschur$, and we have $\schglnq=e\bigschur e$.

\begin{lemma}\label{idempotents fixed by involution}
    Both the idempotent $e$ and the $e_{\diag(\lambda)}$ are fixed under the involution.
\end{lemma}

\begin{proof}
    By definition, $e$ is the sum of the identity elements of each of the summands in \cref{bigschur definition}. But, by \cref{idempotents live in our algebra}, these are precisely the $e_{\diag(\lambda)}$ for $\lambda\in\Lambda_0(N,n)$. It thus suffices to show that the $e_{\diag(\lambda)}$ are preserved under $i$. But by Claim (2) of \cref{stuff from dengyang} we have that $i$ sends $e_A$ to $e_{A^t}$.
\end{proof}

Using these we can reduce to our Schur Algebra.

\begin{theorem}\label{schur algebra affine cellular}
    $\schglnq$ is idempotent affine cellular. For each $k$, $B_k$ is of the form $\ms{Z}[X_1,\dots, X_{m_k}, X_{i_1}^{-1}, X_{i_{n_k}}^{-1}]$ for some $\left\{i_1,\dots,i_{n_k}\right\}\subseteq\left\{1,\dots, m_k\right\}$. The ideals $eJ_k'e$ are generated by $l_k$.
\end{theorem}

\begin{proof}
    Here we follow Theorem 5.7 of \cite{deng2016affine}, which again only considered the more restrictive $q=1$ case.

    We know $\bigschur$ is idempotent affine cellular by \cref{stuff from dengyang}. We want to apply \cref{idempotent preserves affine cellular}. But we have shown that the necessary conditions hold in \cref{idempotents fixed by involution} and \cref{idempotents live in our algebra}. The form of the $B_k$ follows also follows via \cref{idempotent preserves affine cellular} from claim (5) of \cref{stuff from dengyang}.
\end{proof}

To conclude, we make the following algebraic observation.

\begin{lemma}\label{B_s are nice}
    Let $R$ be a commutative ring. Fix some $m\in\mb{N}$ and some set \\$\left\{i_1,\dots,i_n\right\}\subseteq\left\{1,\dots,m\right\}$.
    \begin{enumerate}
        \item If $R$ is a domain, then $\jacrad(R[X_1,\dots, X_{m}, X_{i_1}^{-1}, X_{i_{n}}^{-1}])=0$.
        \item If $\gldim(R)<\infty$, then $\gldim(R[X_1,\dots, X_{m}, X_{i_1}^{-1}, X_{i_{n}}^{-1}])<\infty$.
    \end{enumerate}
\end{lemma}

\begin{proof}
    We show the first result by proving that if $A$ is a domain then $A[x]$ and $A[x,x^{-1}]$ are domains with zero Jacobson radical. That $A[x]$ is a domain follows as the leading coefficients of any nonzero polynomials must multiply to give a nonzero leading coefficient for the product. To see that $\jacrad(A[x])=0$, note that if $p(x)\in\jacrad(A[x])$ then $1+xp(x)$ must be a unit, so $p(x)=0$. The logic for $A[x,x^{-1}]$ is identical except that we now must consider $1+x^kp(x)$ for $k$ large enough that $x^kp(x)$ only has terms of strictly positive degree.

    The second follows from the Hilbert Syzygy Theorem, which gives the polynomial case (see \cite{rotman2009homological} Theorem 8.37). The case for Laurent series follows as localisation is exact and preserves projective modules. Alternatively, it is a special case of \cite{mcconnell2001noetherian}, Theorem 7.5.3 (iii) and (iv).
\end{proof}

Thus we have the result we sought.

\begin{theorem}\label{schur algebra finite global dimension}
    $\gldim(\schglnco)<\infty$.
\end{theorem}

\begin{proof}
    By \cref{schur algebra affine cellular}, $\schglnq$ is idempotent affine cellular, and hence by \cref{extension preserves affine cellular} so is $\schglnco$. Thus by \cref{B_s are nice} we may apply \cref{finite global dimension} to obtain our result.
\end{proof}

In fact, reading the details of the cited proofs can give an explicit upper bound for $\gldim(\schglnco)$ in terms of $n$, but we do not investigate that here.

\section{Generators of the Finite Unipotent Block}\label{finite block}

Recall that we write $\fingp=\redgp(\resfield)$, and within this we have fixed a minimal parabolic subgroup $\finbor$, namely the upper triangular matrices. Similarly, we have the unipotent radical of $\finbor$, written $\finuni$, which is the unipotent upper triangular matrices, and a maximal split torus $\fintor$ in $\finbor$, the diagonal matrices. We shall also need $\finuniopp$, the transpose of $\finuni$, that is, the unipotent lower triangular matrices, and $\finweyl$, the finite Weyl group, which we view as the permutation matrices, with $\finroots$ its simple reflections.

Recall also that we write $\hecfin$ for the group algebra of $\fingp$ over $\coeffs$, and $\finrep$ for the category of $\fingp$-representations over $\coeffs$, that is, modules over $\hecfin$. The unipotent block of $\hecfin$ is the block containing the trivial representation $\triv$, and we write the corresponding direct summand of $\hecfin$ by $\hecfinuni$ and the corresponding direct summand of $\finrep$ by $\uniblockfin$. Recall finally that $\finproj=\ind_{\finbor}^{\fingp}\triv$.

For the rest of the paper, we shall take $n=2$, and we shall further assume that $l$ divides $q+1$ but not $q-1$ or $q$. In this case, we have that, letting $\chi$ be any nontrivial character of $\finuni$, the Gelfand-Graev module is $\ind_{\finuni}^{\fingp}\chi$, and $\funig=(\ind_{\finuni}^{\fingp}\chi)\oplus\parind_{\fintor,\finbor}^{\fingp}\ind_{1}^{\fintor}\triv$.

By transitivity of induction and by commutativity of inflation and induction, we can immediately see that $\parind_{\fintor,\finbor}^{\fingp}\ind_{1}^{\fintor}\triv=\ind_{\finbor}^{\fingp}\infl_{\fintor}^{\finbor}\ind_{1}^{\fintor}\triv=\ind_{\finbor}^{\fingp}\ind_{\finuni}^{\finbor}\infl_{1}^{\finuni}\triv=\ind_{\finuni}^{\fingp}\triv$.

Both $\ind_{\finuni}^{\fingp} \chi$ and $\ind_{\finuni}^{\fingp}\triv$ are projective since $\finuni$ is a $p$-group and $\ind$ preserves projective modules. Hence, $\funig$ is projective.

For our choice of $n$ and $l$, the structure of the unipotent block is well-understood.

\begin{lemma}[{\cite[Theorem 4.2]{ackermann2006loewy}}]\label{structures of projectives}
    $\uniblockfin$ has two simple representations $L_1$ and $L_2$, which are the unique simple quotients of the two projective indecomposable representations $P_1$ and $P_2$, which have the following subrepresentation lattices respectively:

    
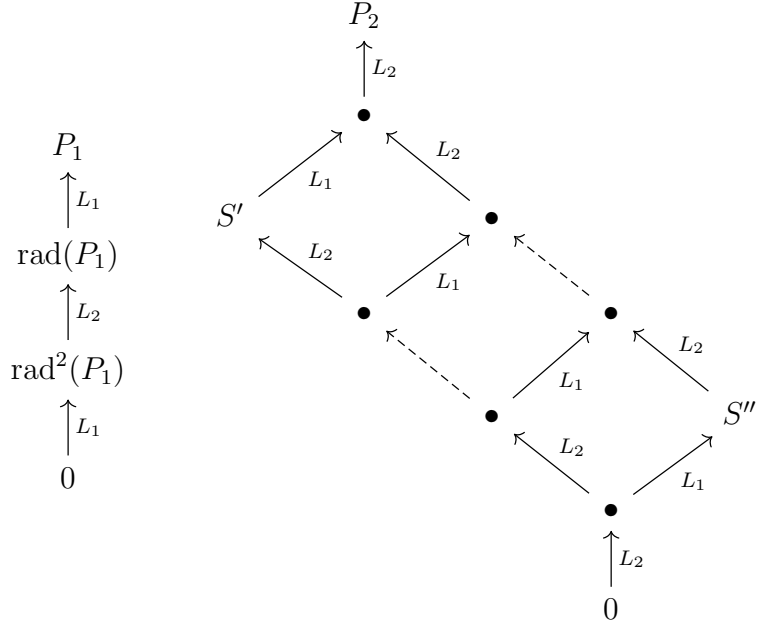
\[
    \begin{tikzcd}
        {P_1} \\
        \rad(P_1) \\
        \rad^2(P_1) \\
        0
        \arrow["{L_1}"', from=2-1, to=1-1]
        \arrow["{L_2}"', from=3-1, to=2-1]
        \arrow["{L_1}"', from=4-1, to=3-1]
    \end{tikzcd} \qquad
    \begin{tikzcd}
        & {P_2} \\
        & \bullet \\
        S' && \bullet \\
        & \bullet && \bullet \\
        && \bullet && S'' \\
        &&& \bullet \\
        &&& 0
        \arrow["{L_2}"', from=2-2, to=1-2]
        \arrow["{L_1}"', from=3-1, to=2-2]
        \arrow["{L_2}"', from=3-3, to=2-2]
        \arrow["{L_2}"', from=4-2, to=3-1]
        \arrow["{L_1}"', from=4-2, to=3-3]
        \arrow[dashed, from=4-4, to=3-3]
        \arrow[dashed, from=5-3, to=4-2]
        \arrow["{L_1}"', from=5-3, to=4-4]
        \arrow["{L_2}"', from=5-5, to=4-4]
        \arrow["{L_2}"', from=6-4, to=5-3]
        \arrow["{L_1}"', from=6-4, to=5-5]
        \arrow["{L_2}"', from=7-4, to=6-4]
    \end{tikzcd}
    \]
\end{lemma}

Note that we have labelled certain subrepresentations, which we will be referring to later.

\begin{corollary}\label{quiver algebra description}
    There is some positive integer $m$ such that the endomorphism algebra of $P_1\oplus P_2$ is generated by the morphisms
    \[
    \begin{tikzcd}
        P_1 & P_2
        \arrow["\alpha", curve={height=-12pt}, from=1-1, to=1-2]
        \arrow["\beta", curve={height=-12pt}, from=1-2, to=1-1]
        \arrow["\gamma", from=1-2, to=1-2, loop, in=325, out=35, distance=10mm]
    \end{tikzcd}
    \]
    which satisfy the relations $\gamma\beta=0$, $\alpha\gamma=0$, and $\beta\alpha=\gamma^m$.
\end{corollary}

\begin{proof}
    We can see from \cref{structures of projectives} that the maps between $P_1$ and $P_2$ are spanned by firstly the map $\alpha:P_1\rightarrow P_2$ with image $S''$, secondly the map $\beta:P_2\rightarrow P_1$ with image $\rad(P_1)$, and finally the map $\gamma:P_2\rightarrow P_2$ with image $S'$. These then have the relations exactly as above.
\end{proof}

\begin{corollary}
    Each $\End(P_i)$ is at least two-dimensional.
\end{corollary}

\begin{proof}
    Each $P_i$ has its identity map, and the maps $\alpha\beta$ and $\beta\alpha$ respectively.
\end{proof}

To obtain the rest of the structure of $\hecfinuni$, we shall find an explicit description of $P_1$ and $L_1$.

\begin{lemma}
    $\finproj$ is projective, indecomposable, and lies in $\uniblockfin$. It has a unique trivial submodule and a unique trivial quotient.
\end{lemma}

\begin{proof}
    Since $l$ does not divide $q$ or $q-1$, the trivial representation $\triv$ on $\finbor$ is projective, and hence $\finproj$ is projective.
    
    By Frobenius reciprocity, we have that $\Hom_{\fingp}(\finproj,\finproj)=\Hom_{\finbor}(\triv,\res_{\finbor}^{\fingp}\finproj)$. Applying the Mackey decomposition and using the Bruhat decomposition $\fingp= \finbor \sqcup \finbor \stwotwo{0}{1}{1}{0} \finbor$ we hence obtain 
    \[\Hom_{\fingp}(\finproj,\finproj)=\Hom_{\finbor}(\triv,\triv)\oplus\Hom_{\finbor}(\triv,\ind_{\fintor}^{\finbor}\triv)=\Hom_{\finbor}(\triv,\triv)\oplus\Hom_{\fintor}(\triv,\triv).\]
    Hence $\End(\finproj)$ is two-dimensional.
    
    By Frobenius reciprocity again, we get $\Hom_{\fingp}(\finproj, \triv)=\Hom_{\finbor}(\triv,\triv)$, and so $\finproj$ has a trivial quotient. As $\triv$ is simple it is some $L_i$, and so we may lift the quotient $\finproj\rightarrow L_i$ to a surjective map $\finproj\rightarrow P_i$, which splits as $P_i$ is projective. But as $\End(P_i)$ is at least two-dimensional, it must be all of $\End(\finproj)$, and so in particular $\finproj=P_i$ is indecomposable.
    
    $\finproj$ lies in $\uniblockfin$ as it is indecomposable and has a trivial submodule.
\end{proof}

\begin{lemma}
    The projective $P_2$ occurs once in an indecomposable decomposition of $\ind_{\finuni}^{\fingp} \chi$, and $P_1$ does not occur at all. $L_2$ is cuspidal.
\end{lemma}

\begin{proof}
    \cite{ackermann2006loewy} Corollary 4.3 identifies $P_2$ and $L_2$ with respectively what they call the Steinberg-PIM and $D_{\resfield}(1,1^2)$. The discussion following Proposition 2.6 of the same then says that the Steinberg-PIM occurs with multiplicity one in the indecomposable decomposition of the Gelfand-Graev module, and that no other projective indecomposable modules in the unipotent block occur in the same. Finally, Proposition 2.2 of the same states that $D_{\resfield}(1,1^2)$ is cuspidal.
\end{proof}

\begin{corollary}
    We have $\finproj=P_1$ and $L_1=\triv$.
    
    Furthermore, $P_1$ occurs with multiplicity 1 and $P_2$ with multiplicity zero in any direct sum decomposition of $\ind_{\finuni}^{\fingp} \triv$.
\end{corollary}

\begin{proof}
    By Frobenius reciprocity, $\Hom_{\fingp}(\triv,\ind_{\finuni}^{\fingp} \chi)=\Hom_{\finuni}(\triv,\chi)=0$, so there are no trivial submodules in $\ind_{\finuni}^{\fingp} \chi$. Thus $L_2$ cannot be $\triv$, so $L_1=\triv$, and since we know $\finproj$ is an indecomposable projective with a unique trivial submodule and quotient we thus know $\finproj=P_1$ by \cref{structures of projectives}.
    
    Similarly, Frobenius Reciprocity also gives us $\Hom_{\fingp}(\triv,\ind_{\finuni}^{\fingp} \triv)=\Hom_{\finuni}(\triv,\triv)=\coeffs$, so there is one submodule of $\ind_{\finuni}^{\fingp}\triv$ isomorphic to $\triv$. Thus, as $\ind_{\finuni}^{\fingp}\triv$ is projective, a direct sum decomposition of it into indecomposable projective modules must contain a single direct summand isomorphic to $P_1$.

    As $L_2$ is cuspidal and $\ind_{\finuni}^{\fingp}\triv$ is parabolically induced from a proper parabolic subgroup, we have that $\Hom(L_2,\ind_{\finuni}^{\fingp} \triv)=0$, and so $\ind_{\finuni}^{\fingp} 1$ contains no submodules isomorphic to $L_2$, and so no direct summands isomorphic to $P_2$.
\end{proof}

\begin{corollary}\label{projective decomposition}
   $\hecfinuni$ can be decomposed as a direct sum of $P_1$ and $q-1$ copies of $P_2$.
\end{corollary}

\begin{proof}
    $\dim(L_1)=1$ and $\dim(P_1)=q+1$, from which we may conclude from \cref{structures of projectives} that $\dim{L_2}=q-1$.
\end{proof}

\begin{definition}
    Write $\fourg=\triv\oplus P_1$.

    Let $\finanni$ be the annihilator of $P_1$ in $\hecfin$. Then put $\fannig=\funig/\finanni\funig$.
\end{definition}

We make some immediate observations.

 \begin{proposition}
    $\fourg$ lives in $\uniblockfin$, and so $\finanni$ contains everything in the non-unipotent blocks. Thus, $\fannig$ lies in $\uniblockfin$.
 \end{proposition}

\begin{proof}
    Both summands of $\fourg$ are unipotent as they are $P_1$ and $L_1$ above. The remaining claims follow immediately.
\end{proof}

We now seek an explicit description of $\finanni$. Henceforth we fix a decomposition $\hecfinuni\cong\hecfin e_1\oplus\oplus \bigoplus_{i=2}^{q}\hecfin e_i$, where $\hecfin e_i\cong P_2$ for all $i\geq 2$. We thus get a choice of $P_1$ and $P_2$ in $\hecfin$, with corresponding idempotents $e_1$ and $e_2$, and we consider the homomorphisms $\alpha$, $\beta$, $\gamma$ of \cref{quiver algebra description} as elements of $\hecfin e_1\oplus \hecfin e_2$ via $\Hom(\hecfin e_i,\hecfin e_j)\cong e_i\hecfin e_j$. Write $f_i$ and $g_i$ for a choice of elements of $e_i\hecfin e_2$ and $e_2\hecfin e_i$ respectively that give inverse isomorphisms $\hecfin e_i \cong P_2$, for $i\geq 3$.

\begin{lemma}
    $\finanni$ is generated as an ideal by the non-unipotent blocks and the element $\gamma$.
\end{lemma}

\begin{proof}
    We already know $\finanni$ contains everything in the non-unipotent blocks. We have that $\hecfinuni\cong\bigoplus_{i,j=1}^qe_i\hecfin e_j$. Hence by \cref{quiver algebra description}, we obtain a basis of the unipotent block given by $e_1, \alpha, \alpha g_j, \beta, f_i\beta, \alpha\beta, \gamma^k, f_i\gamma^k, \gamma^kg_j\ f_i\gamma^kg_j$ for $3\leq i,j \leq q$ and $0\leq k\leq m$. Note that here we take $\gamma^0=e_2$.
    
    We first observe that $\gamma P_1=\gamma\hecfin e_1=\gamma e_2\hecfin e_1$. But $e_2\hecfin e_1$ is spanned by $\beta$, and $\gamma\beta=0$, so $\gamma P_1=0$ and hence $\gamma\in\finanni$. As annihilators are two-sided ideals, we thus get that $\finanni$ contains the subspace with basis $\gamma^k, f_i\gamma^k, \gamma^kg_j\ f_i\gamma^kg_j$ for $3\leq i,j \leq q$ and $1\leq k\leq m$.

    Now consider some $\delta\in\hecfinuni$ not in this span. We wish to show that then it does not lie in $\finanni$. It suffices to assume it is a span of $e_1, \alpha, \alpha g_j, \beta, f_i\beta, \alpha\beta, e_2, f_i, g_j, f_ig_j$ for $3\leq i,j \leq q$.

    Now, $\alpha g_jf_ie_2g_jf_i\beta e_1 P_1=\alpha\beta e_1\hecfin e_1\ni\alpha\beta$, and so $\alpha g_jf_ie_2g_jf_i\beta e_1\notin \finanni$. But as $\finanni$ is a two-sided ideal, this means no sub-product of this product is in $\finanni$, so none of $e_1, \alpha, \alpha g_j, \beta, f_i\beta, \alpha\beta, e_2, f_i, g_j, f_ig_j$ are in $\finanni$.

    But $P_1\cong\bigoplus_{i=1}^q e_i\hecfin e_1$, and so we may consider separately the cases $\delta\in e_i\hecfin e_j$ for $1\leq i,j\leq q$. But the subspace of $e_i\hecfin e_j$ containing all possible values of $\delta$ (that is, in the span of $e_1, \alpha, \alpha g_j, \beta, f_i\beta, \alpha\beta, e_2, f_i, g_j, f_ig_j$ for $3\leq i,j \leq q$) is one-dimensional and spanned by an element not in $\finanni$, except when $i=j=1$. Hence we only need to consider $\delta=\lambda_1e_1+\lambda_2\alpha\beta$. But then $\delta \alpha\beta=\lambda_1\alpha\beta$, so for $\delta$ to annihilate $\alpha\beta\in P_1$ we need $\lambda_1=0$. But then $\delta e_1=\lambda_2\alpha\beta$ so no nonzero such $\delta$ annihilates $P_1$.
\end{proof}

\begin{lemma}\label{structure of fannig}
    $\fannig\cong P_1\oplus \rad(P_1)$.
\end{lemma}

\begin{proof}
    We know $\funig$ contains one copy each of $P_1$ and $P_2$. Hence all its other projective summands $N$ live outside the unipotent block, and so satisfy $\finanni N=N$ and hence $N/\finanni N=0$. But by definition $\finanni P_1=0$, that is, $P_1/\finanni P_1=P_1$. Finally, using the notation of \cref{structures of projectives}, $\finanni P_2 = cP_2 = S'$, so $\rad(P_1)\cong P_2/S'=P_2/\finanni P_2$, and hence $\fannig=\funig/\finanni\funig \cong P_1\oplus \rad(P_1)$.
\end{proof}

\begin{corollary}\label{finite generators are derived equivalent}
    There is a short exact sequence 
    \[\begin{tikzcd}
        0 & {\rad(P_1)} & {P_1} & \triv & 0
        \arrow[from=1-1, to=1-2]
        \arrow[from=1-2, to=1-3]
        \arrow[from=1-3, to=1-4]
        \arrow[from=1-4, to=1-5]
    \end{tikzcd}\]
    two entries of which are the direct summands $P_1$ and $\rad(P_1)$ of $\fourg$, and two of which are the direct summands $P_1$ and $\triv$ of $\fannig$.
\end{corollary}

Consider now the quadratic extension $k'$ of $\resfield$ gained by adjoining a square root of a non-square element $\epsilon$, and let $x+y\sqrt{\epsilon}$ for $x,y\in\resfield$ generate the $l$-torsion in $k'^{\times}$. Note that, as there is no $l$-torsion in $\resfield^{\times}$ (by our assumption that $l$ does not divide $q-1$), but there is nontrivial $l$-torsion in $k'^{\times}$ (as $l$ does divide $q+1$ and hence $q^2-1$), we have that $x$ and $y$ are both nonzero: if not, then $(x+\sqrt{\epsilon}y)^2$ lies in $\resfield$, and hence $x+\sqrt{\epsilon}y$ has order dividing $2(q-1)$, which is not divisible by $l$, which is a contradiction. Write $l^r$ for the order of $x+\sqrt{\epsilon}y$ in $k'^{\times}$, noting that $r>0$ by the previous discussion.

We shall also make use of the following two conjugacy classes in $\fingp$: the conjugacy class in $\fingp$ of $\stwotwo{1}{1}{0}{1}$ will be denoted $C_1$, and the conjugacy class of $\stwotwo{x}{y}{\epsilon y}{x}$ will be denoted $C_2$. Also write $1=\stwotwo{1}{0}{0}{1}$ for the identity matrix in $\fingp$.

We make the initial observation that
\[C_1=\left\{\twotwo{1}{0}{a}{1}\middle|a\in\resinv\right\}\cup \left\{\twotwo{1-ab}{a}{-ab^2}{1+ab}\middle|a\in\resinv,b\in\resfield\right\}\]
and 
\[C_2=\left\{\twotwo{b}{\frac{\epsilon y^2-(x-b)^2}{a}}{a}{2x-b}\middle|a\in\resinv,b\in\resfield\right\}.\]

To see this, note that, for two-by-two matrices, any noncentral matrices sharing a characteristic polynomial must lie in the same conjugacy class. Hence, as the matrices listed have characteristic polynomials $(X-1)^2$ and $(X-x)^2-\epsilon y^2$ respectively, and are not central, they do indeed lie in $C_1$ and $C_2$ respectively. But it is known (see for example \cite{digne1991representations}, Chapter 15 Table 1) that the sizes of $C_1$ and $C_2$ are $q^2-1$ and $q(q-1)$ respectively, so we have in fact obtained the entire class.

We finally write $Z_{\lambda}=\sum_{g\in C_1}g+\sum_{g\in C_2}g+\lambda 1$.

\begin{lemma}\label{What is Z}
    There exists some (necessarily unique) $\lambda\in R$ such that we can take $\gamma$ to be $e_2Z_{\lambda}$.
\end{lemma}

\begin{proof}
   By \cite{paige2014projective}, Proposition 2.6, $P_2$ can be inflated to a representation over $\bar{\mathbb{Z}_l}$, which, when extended to $\bar{\mathbb{Q}_l}$, decomposes as a direct sum of irreducible representations: the Steinberg representation $\pi_0$ and each of the supercuspidal lifts $\pi_i$ of $L_2$, for $i$ in the range $0<i\leq\frac{l^r-1}{2}$.

   From \cite{paige2014projective}, the final remarks of Section 2 on page 363 and the opening remarks of Section 4 on page 368, there is an injective homomorphism $\psi:\End(P_2)\hookrightarrow\mathbb{C}^{\oplus\frac{l^r+1}{2}}$, given by sending the endomorphism $\upsilon$ to a tuple $(u,u_i)$  where $u$ and $u_i$ are the induced actions of $\upsilon$ on $\pi_0$ and $\pi_i$ respectively (necessarily scalar as $\pi_0$ and $\pi_i$ are irreducible).
   
   Furthermore, by those same remarks, the natural map $\phi:Z(\hecfin)\rightarrow\End(P_2)$, sending an element of the centre of the group algebra to its action by multiplication, is surjective. 

   Then, by Theorem 4.11 and Remark 1 in section 4 of \cite{paige2014projective}, we see that $\End(P_2)$ is generated as an algebra by $\phi(Y)$ for some $Y\in Z(\hecfin)$, and that $\psi\phi(Y)=(2,\zeta^i+\zeta^{iq})=(2,\zeta^i+\zeta^{-i})$, for $\zeta$ a primitive $l^r$th root of unity.

   To find an explicit description for a $Y$ with this tuple, we make repeated use of Lemma 4.1 from \cite{paige2014projective}, which says that, if $C$ is a conjugacy class in $\fingp$, then 
   
   \[\phi(\sum_{g\in C}g)=|C|(\frac{\Tr(\pi_0(C))}{\dim\pi_0},\frac{\Tr(\pi_i(C))}{\dim\pi_i}).\]
   
   We combine this with the character table information from \cite{digne1991representations} Chapter 15 Table 1, which says that

    {\centering
    \begin{tabular}{|c||c|c|c|}
        \hline
        Cong. class & $1$ & $C_1$ & $C_2$ \\
        \hline
        Size & $1$ & $(q+1)(q-1)$ & $q(q-1)$ \\
        $\Tr(\pi_0)$ & $q$ & $0$ & $-1$ \\
        $\Tr(\pi_i)$ & $q-1$ & $-1$ & $-\zeta^i-\zeta^{iq}$ \\
        \hline
   \end{tabular}
   
   }

   Using this, and that $\Tr(\pi(1))=\dim\pi$ for any complex representation $\pi$, we get that $\phi(\sum_{g\in C_1}g)=(0,-(q+1))$, that $\phi(\sum_{g\in C_2}g)=(-(q-1),-q(\zeta^i+\zeta^{iq}))$, and that $\phi(1)=(1,1)$.

   Thus we find that a valid choice is $Y=\frac{1}{q}\sum_{g\in C_1}g-\frac{1}{q}\sum_{g\in C_2}g+\frac{q+1}{q}1$, as then $\phi(Y)=(2,\zeta^i+\zeta^{iq})$, and as $Y$ is a sum of sums over conjugacy classes it is central. However, as in our field $q+1=0$ and $q$ is invertible, we may simplify this to $Y=-\sum_{g\in C_1}g+\sum_{g\in C_2}g$.

   Now, as $Y$ is central, we have that $Y$ and $e_2Ye_2=e_2Y$ have the same right action on $P_2=\hecfin e_2$, so the right action of $e_2Y$ also generates $\End(P_2)$.
   
   Now, we also know that the space $\End(P_2)=e_2\hecfin e_2$ is spanned by $e_2$ and the positive powers of $\gamma$. Thus there exists some unique $\lambda$ such that $e_2Z_{\lambda}=e_2(Y+\lambda 1)$ is a linear combination of strictly positive powers of $\gamma$. But by definition $\gamma$ is any arbitrary representative of $\jacrad(\hecfin)/\jacrad(\hecfin)^2$, and so any linear combination of strictly positive powers of a choice of $\gamma$ is also a valid choice for $\gamma$. Hence $e_2Z_{\lambda}$ is a valid choice of $\gamma$.
\end{proof}

Write $e$ for the central idempotent of the unipotent block. Note that $\frac{1}{|\finbor|}\sum_{g\in\finbor}g$ is idempotent and has $\hecfin\frac{1}{|\finbor|}\sum_{g\in\finbor}g\cong P_1$, and so we may henceforth take this to be our choice for $e_1$. We now deduce the value of $\lambda$ in \cref{What is Z}.

\begin{lemma}\label{generator of annihilator is almost zero}
    $Z_{q-1}=\sum_{i}\mu_ic_ix_iz_i+\sum_{j}\nu_jd_j\stwotwo{0}{1}{1}{0}y_j$ for some $c_i,d_j\in\finuni$, some $x_i,y_j\in\fintor$, some $z_i\in\finuniopp$, and some $\mu_i,\nu_j\in\coeffs$, such that $\sum_i \mu_ic_i=0$ and $\sum_j\nu_jd_j=0$.
\end{lemma}

\begin{proof}
    We proceed by factorising each $g$ in $C_1\cup C_2$ into one of the forms $\mu cxz$ or $\nu d\stwotwo{0}{1}{1}{0}y$ for $c,d\in\finuni$, $x,y\in \fintor$, $z\in\finuniopp$, and $\mu,\nu\in\coeffs$. We shall make repeated use of the following formula:

    \begin{align*}
        \twotwo{w}{x}{y}{z}= \begin{cases}
            \twotwo{1}{xz^{-1}}{0}{1}\twotwo{w-xyz^{-1}}{0}{0}{z}\twotwo{1}{0}{yz^{-1}}{1} \quad \text{if} \ z\neq0 \\[5pt]
            \twotwo{1}{wy^{-1}}{0}{1}\twotwo{0}{1}{1}{0}\twotwo{y}{0}{0}{x} \quad \text{if} \ z=0, \ y\neq 0
        \end{cases}
    \end{align*}
    observing that these are respectively of the two above forms.
    
    We start with $g\in C_1$. First, note that $\stwotwo{1}{0}{a}{1}\in \finuniopp$ for all $a\in\resinv$, and so summing over the $q-1$ terms in $C_1$ of this form we get $\sum_{a\in\resinv}z_{1,0,a}$ for $z_{1,0,a}=\stwotwo{1}{0}{a}{1}\in\finuniopp$.

    Next, let $a\in\resinv$, $b\in\resfield$ be such that $1+ab\neq0$. Then
    \[\twotwo{1-ab}{a}{-ab^2}{1+ab}\in\twotwo{1}{a(1+ab)^{-1}}{0}{1}\fintor\finuniopp.\]

    Now, observe that $a(1+ab)^{-1}=(a^{-1}+b)^{-1}$, and so we have two cases:
    
    If $b=0$, then in fact we just have 
    \[\twotwo{1-ab}{a}{-ab^2}{1+ab}=\twotwo{1}{a}{0}{1}\] 
    Hence summing the $q-1$ terms we get by varying $a$ over $\resinv$ gives $\sum_{a\in\resinv}\stwotwo{1}{a}{0}{1}$.

    If $b\neq0$, then $1+ab=0$ precisely when $a=-b^{-1}$, so fixing $b$ and varying $a$ over $\resinv\backslash\{-b^{-1}\}$ means $(a^{-1}+b)^{-1}$ takes every value $a'$ in $\resfield$ a single time (by setting $a=(a'-b)^{-1})$ except $0$ (which would need $a=-b^{-1}$) and $b^{-1}$ (which would need $a=0$). Hence summing these $q-2$ terms gives $\sum_{a'\in\resinv\backslash\{b^{-1}\}}\stwotwo{1}{a'}{0}{1}x_{1,b,a'}z_{1,b,a'}$ for some $x_{1,b,a'}\in\fintor$ and $z_{1,b,a'}\in\finuniopp$.

    Meanwhile, if $1+ab=0$, we observe that necessarily $b\neq0$, and that if we fix such a $b$ then there is exactly one $a$ satisfying this relation, namely $a=-b^{-1}$. Now, we have that
    \[\twotwo{1-ab}{a}{-ab^2}{1+ab}\in\twotwo{1}{(1-ab)(-ab^2)^{-1}}{0}{1}\twotwo{0}{1}{1}{0}\finuniopp.\]
    Simplifying gives $(1-ab)(-ab^2)^{-1}=2b^{-1}$, and varying $b$ over $\resinv$ means $2b^{-1}$ takes every value $a'\in\resinv$ a single time (by setting $b=2a'^{-1}$). Hence summing these $q-1$ terms gives $\sum_{a'\in\resinv}\stwotwo{1}{a'}{0}{1}\stwotwo{0}{1}{1}{0}y_{1,a'}$.

    Now we consider $g\in C_2$. Let $a\in\resinv$ and $b\in\resfield$, and suppose first that $2x-b\neq 0$. Then
    \[\twotwo{b}{\frac{\epsilon y^2-(x-b)^2}{a}}{a}{2x-b}\in\twotwo{1}{\frac{\epsilon y^2-(x-b)^2}{a(2x-b)}}{0}{1}\fintor\finuniopp.\]
    Now, note that $\epsilon y^2-(x-b)^2\neq0$ as $\epsilon$ is by assumption not a square. Hence fixing $b$ and varying $a$ over $\resinv$ means $\frac{\epsilon y^2-(x-b)^2}{a(2x-b)}$ takes every value $a'\in\resinv$ a single time (by setting $a=\frac{\epsilon y^2-(x-b)^2}{a'(2x-b)}$). Hence summing these $q-1$ terms gives $\sum_{a'\in\resinv}\stwotwo{1}{a'}{0}{1}x_{2,b,a'}z_{2,b,a'}$ for some $x_{2,b,a'}\in\fintor$ and $z_{2,b,a'}\in\finuniopp$.

    The remaining case is when $2x-b=0$, noting there is exactly one such $b$, namely $2x$, and that it is not zero, since $x$ cannot be zero and $l\neq2$. Then
    \[\twotwo{b}{\frac{\epsilon y^2-(x-b)^2}{a}}{a}{2x-b}\in\twotwo{1}{ba^{-1}}{0}{1}\twotwo{0}{1}{1}{0}\fintor\]
    and so varying $a$ over $\resinv$ means $ba^{-1}$ takes every value $a'\in\resinv$ a single time (by setting $a=ba'^{-1}$). Hence summing these $q-1$ terms gives $\sum_{a'\in\resinv}\stwotwo{1}{a'}{0}{1}\stwotwo{0}{1}{1}{0}y_{2,a'}$.

    Now, we recall that we have $Z_{q-1}=-\sum_{g\in C_1}g+\sum_{g\in c_2}g+(q-1)1$. Hence, putting all our results together, we have that
    \begin{align*}
        Z_{q-1}=&-\sum_{a\in\resinv}z_{1,0,a}-\sum_{a\in\resinv}\stwotwo{1}{a}{0}{1}-\sum_{b\in\resinv}\sum_{a'\in\resinv\backslash\{b^{-1}\}}\stwotwo{1}{a'}{0}{1}x_{1,b,a'}z_{1,b,a'} \\
        &-\sum_{a'\in\resinv}\stwotwo{1}{a'}{0}{1}\stwotwo{0}{1}{1}{0}y_{1,a'}+\sum_{b\in\resfield\backslash\{2x\}}\sum_{a'\in\resinv}\stwotwo{1}{a'}{0}{1}x_{2,b,a'}z_{2,b,a'} \\
        &+\sum_{a'\in\resinv}\stwotwo{1}{a'}{0}{1}\stwotwo{0}{1}{1}{0}y_{2,a'}+(q-1)\stwotwo{1}{0}{0}{1}
    \end{align*}
    which is of the form $\sum_{i}\mu_i c_ix_iz_i+\sum_{j}\nu_jd_j\stwotwo{0}{1}{1}{0}y_j$ for some $c_i,d_j\in\finuni$, some $x_i,y_j\in\fintor$, some $z_i\in\finuniopp$, and some $\mu_i,\nu_j\in\coeffs$. It remains to check that $\sum_i \mu_i c_i=0$ and $\sum_j \nu_jd_j=0$.

    Considering first $\sum_j \nu_jd_j$, we can see that we get
    \[-\sum_{a'\in\resinv}\stwotwo{1}{a'}{0}{1}+\sum_{a'\in\resinv}\stwotwo{1}{a'}{0}{1}=0.\]

    Now we consider $\sum_i \mu_i c_i$. This is
    \begin{align*}
        &-\sum_{a\in\resinv}\stwotwo{1}{0}{0}{1}-\sum_{a\in\resinv}\stwotwo{1}{a}{0}{1}-\sum_{b\in\resinv}\sum_{a'\in\resinv\backslash\{b^{-1}\}}\stwotwo{1}{a'}{0}{1} \\
        &+\sum_{b\in\resfield\backslash\{2x\}}\sum_{a'\in\resinv}\stwotwo{1}{a'}{0}{1}+(q-1)\stwotwo{1}{0}{0}{1} \\
        = \ &(-(q-1)+(q-1))\stwotwo{1}{0}{0}{1}+\sum_{a'\in\resinv}(-1-(q-2)+(q-1))\stwotwo{1}{a'}{0}{1} \\
        = \ &0.
    \end{align*}
\end{proof}

The above calculation allows us to find the value of $\lambda$ in \cref{What is Z}.

\begin{lemma}\label{lambda is 1}
    $\gamma=e_2Z_{q-1}$.
\end{lemma}

\begin{proof}
    We shall directly calculate $(e-e_1)Z_{q-1}P_1$ and observe that it is zero. As then $(e-e_1)(Z_{q-1}+\mu 1)P_1=\mu(e-e_1)P_1$ is only zero for $\mu=0$, only $\lambda=q-1$ gives $(e-e_1)Z_{\lambda}\in\finanni$, whence the claim.

    Now, $eP_1=P_1$, so it in fact suffices to show that $(1-e_1)Z_{q-1}P_1=0$. Thus, we need to show for any $g\in\fingp$ that $(1-e_1)Z_{q-1}g\sum_{b\in\finbor}b=0$, that is, that $Z_{q-1}g\sum_{b\in\finbor}b$ is left-$\finbor$-invariant.
    
    By the Bruhat decomposition, we may without loss of generality take  $g=iw$ for $i\in\finbor$ and $w$ either $1$ or $\stwotwo{0}{1}{1}{0}$. But $i$ commutes with $Z_{q-1}$ as the latter is central, and so $Z_{q-1}g\sum_{b\in\finbor}b=iZ_{q-1}w\sum_{b\in\finbor}b$. Thus we may without loss of generality take $i=1$, and show that $Z_{q-1}w\sum_{b\in\finbor}b$ is left-$\finbor$-invariant.

    Now, we can write $\finbor=\finuni\fintor$. Observe first that $\fintor$ commutes with $w$, and also with $Z_{q-1}$ as $Z_{q-1}$ is central. Thus, if $i\in\fintor$, then $iZ_{q-1}w\sum_{b\in\finbor}b=Z_{q-1}wi\sum_{b\in\finbor}b=Z_{q-1}w\sum_{b\in\finbor}b$, so $Z_{q-1}w\sum_{b\in\finbor}b$ is left-$\fintor$-invariant. Thus, it only remains to prove that $w\sum_{b\in\finbor}b$ is left $\finuni$-invariant.

    We shall divide this into two cases depending on the value of $w$, respectively $1$ and $\stwotwo{0}{1}{1}{0}$.
 
    Consider the first case. Then $Z_{q-1}w\sum_{b\in\finbor}b=Z_{q-1}\sum_{b\in\finbor}b$. As $\finuni$ commutes with $Z_{q-1}$ as $Z_{q-1}$ is central, if $i\in\finuni$ then $iZ_{q-1}\sum_{b\in\finbor}b=Z_{q-1}i\sum_{b\in\finbor}b=Z_{q-1}\sum_{b\in\finbor}b$, which is exactly left-$\finuni$-invariance.
    
    For the second case, we use that, by \cref{generator of annihilator is almost zero}, we have $Z_{q-1}=\sum_{i}\mu_ic_ix_iz_i+\sum_{j}\nu_jd_j\stwotwo{0}{1}{1}{0}y_j$ for some $c_i,d_j\in\finuni$, some $x_i,y_j\in\fintor$, some $z_i\in\finuniopp$, and some $\mu_i,\nu_i\in\coeffs$, such that $\sum_i\mu_i c_i=0$ and $\sum_j\nu_j d_j=0$. Now, we have that $\finuniopp w=w\finuni\subseteq w\finbor$, so $z_iw\in w\finbor$. Additionally, $\fintor w=w\fintor\subseteq w\finbor$, so $x_iw, y_jw\in w\finbor$. Thus, 
    \begin{align*}
        Z_{q-1}w\sum_{b\in\finbor}b &=\sum_{i}\mu_ic_ix_iz_iw\sum_{b\in\finbor}b+\sum_{j}\nu_jd_j\stwotwo{0}{1}{1}{0}y_jw\sum_{b\in\finbor}b \\
        &= \sum_{i}\mu_ic_iw\sum_{b\in\finbor}b+\sum_{j}\nu_jd_j\stwotwo{0}{1}{1}{0}w\sum_{b\in\finbor}b \\
        &=0
    \end{align*}
    where the last equality follows as $\sum_i \mu_ic_i=0$ and $\sum_j \nu_jd_j=0$. As in fact \[Z_{q-1}w\sum_{b\in\finbor}b=0\] in this case, it is certainly in particular left-$\finuni$-invariant.
\end{proof}

Henceforth we shall simply write $Z$ for $Z_{q-1}$.

\begin{lemma}\label{c generates annihilator}
    $(e-e_1)Z$ generates the unipotent part of $\finanni$.
\end{lemma}

\begin{proof}
    We have $\gamma =e_2Z=e_2(e-e_1)Z$, so it suffices to show that $(e-e_1)Z\in\finanni$. Now, $(e-e_1)Z=\sum_{i=2}^q e_iZ$. But as $Z$ is central we have $e_iZ=f_ig_iZ=f_iZg_i=f_ie_2Zg_i=f_i\gamma g_i\in\finanni$, so $(e-e_1)Z\in\finanni$.
\end{proof}

We thus have an explicit description of $\finanni$, namely that it is generated by the non-unipotent elements plus a single unipotent element $\gamma=(e-e_1)Z$, where $Z=-\sum_{g\in C_1}g+\sum_{g\in C_2}g+(q-1)1$ for certain conjugacy classes $C_1$ and $C_2$ in $\fingp$, whose elements we have given explicitly. Furthermore, we have an exact sequence \cref{finite generators are derived equivalent} relating the summands of $\fannig$ and $\fourg$. These two facts shall be used to obtain our derived equivalence.

\section{Lifting to the $p$-adic Setting}\label{p-adic lift}

We now use our knowledge of the finite structure to understand how our various $p$-adic generators relate. In order to lift results between the two setting, we first make the following observation relating the finite and $p$-adic unipotent blocks:

\begin{proposition}\label{unipotent induces to unipotent}
    Let $\pi_f\in\finrep$. If $\pi_f\in\uniblockfin$, then $\phorind_{\fingp,\maxcomp}^{\padicgp}\pi_f\in\uniblock$. Conversely, if $\pi_f$ has no summand in $\uniblockfin$, then $\phorind_{\fingp,\maxcomp}^{\padicgp}\pi_f$ has no summand in $\uniblock$.
\end{proposition}

\begin{proof}
    This is \cite{vigneras2003schur}, Lemma D14 ($a_1$) and ($a_2$), noting that the hypothesis Conjecture~$H_3$ in said paper is stated to hold for $G=\glnp$, and noting that $\finproj$ is in the finite unipotent block in our case, and so ours and Vign\'eras's notions of finite unipotent block agree.
\end{proof}

\begin{definition}
    Write $\ourgen=\phorind_{\fingp,\maxcomp}^{\padicgp}\fourg$.
\end{definition}

Immediately from \cref{unipotent induces to unipotent} we get

\begin{proposition}\label{our gen is in the unipotent}
    $\ourgen$ and $\phorind_{\fingp,\maxcomp}^{\padicgp}\fannig$ are both finitely generated and unipotent.
\end{proposition}

\begin{proof}
    $\ourgen$ and $\phorind_{\fingp,\maxcomp}^{\padicgp}\fannig$ are the parahoric induction from $\fourg$ and $\fannig$ respectively, which are unipotent and finitely generated, and parahoric induction preserves being finitely generated.
\end{proof}

We wish to use the exactness of parahoric induction to lift our relations between $\fannig$ and $\fourg$ to relations between $\annigen$ and $\ourgen$. Unfortunately, we run into the issue that $\phorind_{\fingp,\maxcomp}^{\padicgp}\fannig=\unigen/(\phorind_{\fingp,\maxcomp}^{\padicgp}\finanni\funig)$ is not a priori equal to $\annigen=\unigen/(\annihilator \, \phorind_{\fingp,\maxcomp}^{\padicgp}\funig)$. We shall remedy this using our understanding of $\finanni$ and $\fannig$.

\begin{lemma}
    $\annihilator\phorind_{\fingp,\maxcomp}^{\padicgp}\fannig=0$.
\end{lemma}

\begin{proof}
    By \cref{structure of fannig}, $\fannig$ is a direct sum of submodules of $\finproj$. Hence $\phorind_{\fingp,\maxcomp}^{\padicgp}\fannig$ is a direct sum of submodules of $\project$, which are annihilated by $\annihilator$.
\end{proof}

As a consequence we can observe that
\begin{align*}
    \begin{split}
        0 = \annihilator\phorind_{\fingp,\maxcomp}^{\padicgp}\fannig 
        = (\annihilator\unigen+\phorind_{\fingp,\maxcomp}^{\padicgp}\finanni\funig)/(\phorind_{\fingp,\maxcomp}^{\padicgp}\finanni\funig)
    \end{split}
\end{align*}
and hence that 
\begin{align}\label{induced finite is quotient}
    \annihilator\unigen\subseteq\phorind_{\fingp,\maxcomp}^{\padicgp}\finanni\funig.
\end{align}

To show the reverse inclusion, it is simplest to consider how parahoric induction appears in the setting of $\hecglob$-modules. Let $\heccomp$ be the subalgebra of $\hecglob$ supported on $\maxcomp$. The image of the map on $\heccomp$ given by $f\mapsto \indicate_{\propmax}f\indicate_{\propmax}$ may be identified with $\hecfin$ via $g\leftrightarrow\indicate_{g\propmax}$. Hence $\hecfin$-modules may be viewed as $\heccomp$-modules via inflation along this map. Viewing representations of $\fingp$ and $\padicgp$ as modules over $\hecfin$ and $\hecglob$ respectively, compact induction is then exactly $\hecglob\otimes_{\heccomp}-$.

Using these, we may rewrite $\phorind_{\fingp,\maxcomp}^{\padicgp}\finanni\funig$ and $\annihilator\unigen$ as
\begin{align}\label{formula for induced finite generator}
    \begin{split}
        \phorind_{\fingp,\maxcomp}^{\padicgp}\finanni\funig &= \hecglob\otimes_{\heccomp}\finanni\funig
    \end{split}
\end{align}
and
\begin{align}\label{formula for generator}
    \begin{split}
        \annihilator\unigen &= \annihilator\hecglob\otimes_{\heccomp}\funig 
        =\annihilator\otimes_{\heccomp}\hecfin\funig
    \end{split}
\end{align}
respectively.

Hence we wish to show $\hecglob\otimes_{\heccomp}\finanni$ lies inside $\annihilator\otimes_{\heccomp}\hecfin\cong \annihilator$, as subsets of $\hecglob\otimes_{\heccomp}\hecfin\cong\hecglob$. We shall abuse notation twofold: firstly by using the above isomorphism to write $h$ in place of $h\otimes1$ for $h\in\hecglob$, and secondly by writing elements $1\otimes g$ for $g\in\fingp$ simply as $g$. Observe that $1\otimes g=\indicate_{g\propmax}\otimes 1$, that is, under the abuse of notation, $g=\indicate_{g\propmax}$.

We show this using our explicit knowledge of $\finanni$. In the following, we make frequent use of the following subsets 

\begin{itemize}
    \item $\iwaupper$, the unipotent upper triangular matrices in $\iwa$,
    \item $\iwadiag$, the diagonal matrices in $\iwa$,
    \item $\iwalower$, the unipotent lower triangular matrices in $\iwa$, and
    \item $\iwaoppupp$, the transpose of $\iwaupper$.
\end{itemize}

We also recall the elements $e_1=\frac{1}{|\finbor|}\sum_{g\in\finbor}g$ and $Z$ of $\hecfin$. Recall that (by \cref{generator of annihilator is almost zero}) the latter is $Z=\sum_{i}\mu_ic_ix_iz_i+\sum_{j}\nu_jd_j\stwotwo{0}{1}{1}{0}y_j$ for some $c_i,d_j\in\finuni$, some $x_i,y_j\in\fintor$, some $z_i\in\finuniopp$, and some $\mu_i,\nu_i\in\coeffs$, such that $\sum_i\mu_i c_i=0$ and $\sum_j\nu_j d_j=0$. Hence in our notation they become the elements $\frac{1}{|\finbor|}\indicate_{\iwa}$ and $\sum_{i}\mu_i\indicate_{\propmax c_ix_iz_i}+\sum_{j}\nu_j\indicate_{\propmax d_j\stwotwo{0}{1}{1}{0}y_j}$ of $\hecglob$, which we still denote $e_1$ and $Z$ respectively for simplicity.

\begin{lemma}\label{finite annihilator lives in annihilator}
    $\hecglob\otimes_{\heccomp}\finanni\subseteq\annihilator$.
\end{lemma}

\begin{proof}
    This proof is a straightforward generalisation of the proof of \cref{lambda is 1}, where now we consider the extended affine Weyl group $\weyl$ instead of just the finite Weyl group $\finweyl$.

   By \cref{c generates annihilator} we know that $\finanni$ is a 2-sided ideal generated by $(e-e_1)Z$ and elements in the non-unipotent finite blocks. Now, under $\phorind_{\fingp,\maxcomp}^{\padicgp}$, any element of $\finanni$ in a non-unipotent finite block will be lifted to an element of $\hecglob$ in a product of non-unipotent $p$-adic blocks, by \cref{unipotent induces to unipotent}, and hence lie inside $\annihilator$. Thus, it suffices to show $(e-e_1)Z\in\annihilator$. Furtherore, any primitive central idempotent of $\hecfin$ other than the central idempotent $e$ of the unipotent block will also be lifted to an element of $\hecglob$ in a product of non-unipotent blocks, and so will annihilate anything in $\uniblock$. Hence, as $\project$ is in $\uniblock$, we get that $(e-e_1)Z$ annihilates $\project$ if and only if $(1-e_1)Z$ does. Hence, we must show that, for any $g\in\padicgp$, we have $(1-e_1)Z\indicate_{g\iwa}=0$.

    Recall that in $\hecglob$ we have that $e_1=\frac{1}{|\finbor|}\indicate_{\iwa}$. Thus multiplication by $e_1$ on the left is exactly projection onto the subspace of left-$\iwa$-invariant vectors. Thus it suffices to show that $Z\indicate_{g\iwa}$ is left-$\iwa$-invariant. 
    
    By the Iwahori decomposition, we may without loss of generality take $g=iw$ for $w\in \weyl$ and $i\in \iwa$. But the image of $i$ in $\fingp$ commutes with $Z$ as the latter is central, and so $i$ commutes with $Z$, and so $Z\indicate_{g\iwa}=Z\indicate_{iw\iwa}=iZ\indicate_{w\iwa}$. Thus we may without loss of generality take $i=1$, and show that $Z\indicate_{w\iwa}$ is left-$\iwa$-invariant.

    Now, we can write $\iwa=\iwaupper\iwadiag\iwalower$. Observe first that $\iwalower$ is a subset of $\propmax$, and that $Z$ is a linear combination of $\indicate_{\propmax g'}$ in $\hecglob$, which are left-$\propmax$-invariant, so  $Z\indicate_{w\iwa}$ is left-$\iwalower$-invariant. Next, observe that $\iwadiag$ commutes with $\weyl$ as the Weyl group normalises the diagonal elements of $\padicgp$, and $\iwadiag$ also commutes with $Z$ as $Z$ is central in $\fingp$. Thus, if $i\in\iwadiag$, then $iZ\indicate_{w\iwa}=Z\indicate_{wi\iwa}=Z\indicate_{w\iwa}$, so $Z\indicate_{w\iwa}$ is also left-$\iwadiag$-invariant. Thus, it only remains to prove that $Z\indicate_{w\iwa}$ is left $\iwaupper$-invariant.

    We shall divide this into two cases depending on the value of $w$. These are respectively
    $\stwotwo{\varpi^k}{0}{0}{\varpi^l}$ and $\stwotwo{0}{\varpi^k}{\varpi^l}{0}$, for $k,l\in\mb{Z}$.
 
    Consider the first case. There are two sub-cases. If $k\leq l$, then $w\iwaupper w^{-1}\supseteq \iwaupper$, and hence $w\iwa=w\iwaupper\iwa=w\iwaupper w^{-1}w\iwa\supseteq \iwaupper w\iwa=\supseteq w\iwa$. Hence we have equality $\iwaupper w\iwa=w\iwa$. But $\iwaupper$ commutes with $Z$ as $Z$ is central in $\fingp$, so if $i\in\iwaupper$ then $iZ\indicate_{w\iwa}=Z\indicate_{iw\iwa}=Z\indicate_{w\iwa}$, which is exactly left-$\iwaupper$-invariance.
    
    Meanwhile, if $k>l$, we have $Z=\sum_{i}\mu_ic_ix_iz_i+\sum_{j}\nu_jd_j\stwotwo{0}{1}{1}{0}y_j$ as in \cref{generator of annihilator is almost zero}. Now, as $k>l$, we have that $\iwaoppupp w\subseteq w\iwalower\subseteq w\iwa$, so $z_iw\in w\iwa$. Additionally, $\iwadiag w=w\iwadiag\subseteq w\iwa$, so $x_iw\in w\iwa$. Thus, 
    \begin{align*}
        Z\indicate_{w\iwa} &=\sum_{i}\mu_i\indicate_{\propmax c_ix_iz_i}\indicate_{w\iwa}+\sum_{j}\nu_j\indicate_{\propmax d_j\stwotwo{0}{1}{1}{0}y_j}\indicate_{w\iwa} \\
        &= \sum_{i}\mu_ic_i\indicate_{\propmax w\iwa}+\sum_{j}\nu_jd_j\indicate_{\propmax\stwotwo{0}{1}{1}{0}w\iwa} \\
        &=0
    \end{align*}
    where the last equality follows as $\sum_i \mu_ic_i=0$ and $\sum_j \nu_jd_j=0$. As $Z\indicate_{w\iwa}=0$ in this case, it is certainly in particular left-$\iwaupper$-invariant.

    The reasoning for the second case is similar. In this case we have that, if $k<l$, then $w\iwalower w^{-1}\supseteq \iwaupper$, and hence by analogous reasoning to the previous case we get immediate left-$\iwaupper$-invariance. Similarly, in the case $k\geq l$ we have that $\iwaoppupp w\subseteq w\iwaupper\subseteq w\iwa$, and so analogous reasoning again gives left $\iwaupper$-invariance.
\end{proof}

\begin{corollary}\label{annigen induced from fingen}
    $\annigen=\phorind_{\fingp,\maxcomp}^{\padicgp}\fannig$.
\end{corollary}

\begin{proof}
    By \cref{finite annihilator lives in annihilator}, together wtih \cref{formula for induced finite generator} and \cref{formula for generator}, we have that $\phorind_{\fingp,\maxcomp}^{\padicgp}\finanni\funig\subseteq\annihilator\unigen$. But by \cref{induced finite is quotient}, the converse is true. Thus they are equal.
\end{proof}

This, together with \cref{finite generators are derived equivalent}, will suffice for us to produce our derived equivalence.

\section{dg Algebras and their Perfect Complexes}\label{derived categories}

We first recall the definitions of dg algebras and their derived categories. For now, let $A$ denote an arbitrary $R$-algebra.

Let $(\cx{M},d')$ and $(\cx{N},d'')$ be $A$-complexes. The $R$-complex  $\dghom_A(\cx{M},\cx{N})$ of dg morphisms has in degree $n$ the homogeneous $A$-module homomorphisms $f:\cx{M}\rightarrow \cx{N}$ of degree $n$, with differential
    \[
        df:=d''f-(-1)^nfd'
    \]
    for all $f$ of degree $n$.

A morphism of complexes is a dg morphism $f$ of degree $0$ that lies in $\ker(d)$, that is, such that $d''f=fd'$.

A homotopy of morphisms of complexes $f,g:\cx{M}\rightarrow \cx{N}$ is a dg morphism $h$ of degree $-1$ such that $dh=f-g$, that is, such that $f-g=d''h-hd'$. If such an $h$ exists we say $f$ and $g$ are homotopic. This gives an equivalence relation.

Write $\htpymod{A}(\cx{M},\cx{N})$ for the homotopy equivalence classes of morphisms of complexes $f,g:M\rightarrow N$. Then we have that $\nhomlgy(\dghom_A(\cx{M},\cx{N}))=\htpymod{A}(\cx{M},\cx{N}[n])$.

We can extend these notions to a richer class of structures than ordinary algebras:

\begin{definition}
    A dg algebra (short for 'differential graded algebra') over $R$ is an $R$-complex $(B,d)$ that is also a graded algebra over $R$, satisfying the graded Leibniz rule
    \[
    d(fg)=d(f)g+(-1)^nfd(g)
    \]
    for all $f$ of degree $n$.
\end{definition}

We think of an ordinary (ie non-dg) algebra as a dg algebra with all elements having degree $0$.

The way we shall obtain dg algebras is by taking dg endomorphisms of complexes over ordinary algebras: the dg endomorphism algebra $\dgend_A(\cx{M})$ is the dg algebra whose complex is $\dghom_A(\cx{M},\cx{M})$, with multiplication given by componentwise composition.

A (left) dg module over $B$ is an $R$-complex $(\cx{M},d')$ with a graded (left) module action of $B$ of degree $0$, compatible with the $R$-action, such that
\[
    d'(fv)=(df)v+(-1)^nf(d'v)
\]
    for all $f\in B$ of degree n.

Let $(\cx{M},d')$ and $(\cx{N},d'')$ be dg modules. The complex of dg morphisms of dg modules $\dghom_B(\cx{M},\cx{N})$ is the subcomplex of $\dghom_R(\cx{M},\cx{N})$ consisting of those dg morphisms that are also $B$-module homomorphisms. 
    
Morphisms and homotopies of morphisms of dg modules are the dg morphisms of dg modules that are also morphisms and homotopies of morphisms of complexes respectively.

When $B=A$ is an ordinary algebra, the two notions of $\dghom_A(\cx{M},\cx{N})$ agree.
    
Write $\htpymod{B}$ for the triangulated category whose objects are dg modules over $B$ and whose morphisms are homotopy equivalence classes of morphisms of dg modules. The distinguished triangles in $\htpymod{B}$ are the exact sequences that are split as graded modules.

A morphism of dg modules $\cx{M}\rightarrow\cx{N}$ is a quasi-isomorphism if the induced morphism of graded modules $\homlgycx(\cx{M})\rightarrow\homlgycx(\cx{N})$ is an isomorphism.

The derived category of dg modules over $B$, written $\dermod{B}$, is the localisation of $\htpymod{B}$ with respect to the quasi-isomorphisms, with the inherited triangulated structure. Then any short exact sequence in the category of dg modules is part of a unique distinguished triangle in $\dermod{B}$. In particular, if two dg modules in an exact sequence in the category of dg modules both live in a full triangulated subcategory of $\dermod{B}$, then so does the third, and hence such categories are closed under kernels, cokernels, and extensions in the category of dg modules. Furthermore, if $\cx{M}$ is a projective resolution of $M$ in $A\modcat$, then
\[
    \nhomlgy(\dgend(\cx{M}))\cong\Hom_{\dermod{A}}(M, M[n]). 
\] 

\begin{definition}
    We say a a set of objects $G$ of a triangulated category $T$ \emph{classically generates} a triangulated subcategory $T'$ of $T$ if $T'$ is the smallest full triangulated subcategory of $T$ closed under isomorphisms and direct summands and containing $G$. We also write $T'=\langle G\rangle_T$.
\end{definition}

\begin{definition}
    The category of perfect objects in $\dermod{B}$, written $\per(B)$, is $\langle B\rangle_{\dermod{B}}$.
\end{definition}

Observe that, in the case that $B$ is an ordinary algebra $A$, then $\per(A)$ is the full subcategory of $\dermod{A}$ consisting of objects isomorphic to finite length complexes of finitely generated projective $A$-modules.

We shall also need $\bdermod{A}$, the full subcategory of $\dermod{A}$ consisting of objects isomorphic to finite length complexes of finitely generated $A$-modules. This is also a full triangulated subcategory of $\dermod{A}$ closed under direct summands in $\dermod{A}$.

\begin{theorem}\label{triangulated equivalence}
    Let $\mathcal{T}$ be a full triangulated subcategory of $\derivedcat(A)$ that is closed under direct summands, let $M$ be in both $A\modcat$ and $\mathcal{T}$, such that $\langle M\rangle _{\mathcal{T}}=\mathcal{T}$, and let $\cx{M}$ be a projective resolution of $M$ in $A\modcat$. Then there is a triangulated equivalence
    \[
    \mathcal{T}\simeq\per(\dgend(\cx{M})).
    \]
\end{theorem}

\begin{proof}
    We seek to apply Theorem 3.8(b) of \cite{keller2006differential}, so we must check that all the conditions of said theorem hold. Section 3.6 of the same establishes that $\derivedcat(A)$ is what he calls 'algebraic', and hence so is $\mathcal{T}$, as it is a triangulated subcategory. Furthermore, Section 3.5 of the same establishes that $\derivedcat(A)$ is idempotent-complete (as it has arbitrary coproducts), and hence, as $\mathcal{T}$ is closed under direct summands, $\mathcal{T}$ is also idempotent-complete. Finally, $\nhomlgy(\dgend(\cx{M}))\cong\Hom_{\dermod{A}}(M, M[n])$, and so as $\mathcal{T}$ is full we get $\nhomlgy(\dgend(\cx{M}))\cong\Hom_{\mathcal{T}}(M, M[n])$. Thus all the conditions of Theorem 3.8(b) hold.
\end{proof}

\begin{definition}
    An object $M$ in $A\modcat$ is a generator if every object in $A\modcat$ is the quotient of a direct sum of copies of $M$.

    A finitely generated projective generator is called a progenerator.
\end{definition}

We may now use our previous results to draw conclusions about the derived category of the unipotent block $\unihec$ of the global Hecke algebra $\hecglob$ of $\padicgp$. Recall that $\annihec=\hecglob/\annihilator$.

\begin{lemma}\label{mod is proj}
    $\bdermod{\annihec}=\per(\annihec)$.
\end{lemma}

\begin{proof}
    By \cref{Vigneras background}, $\schglnco$ and $\annihec$ are Morita equivalent. Thus, as $\schglnco$ has finite global dimension by \cref{schur algebra finite global dimension}, so does $\annihec$. But $\annihec$ is also Noetherian by \cref{everything is noetherian}. Thus, by \cite{rotman2009homological}, Lemma 7.19 and Proposition 8.16, every finitely generated $\annihec$-module has a finite length resolution by finitely generated projective modules. Thus, by \cite{gelfand2003homological} Lemma III.7.12, every every object $\cx{M}$ of $\bdermod{\annihec}$ is isomorphic to the total complex $T^k=\oplus_{i+j=k}P^{ij}$, where $P^{i\bullet}$ is a choice of such a resolution for $M^i$. But this is an object in the subcategory $\per(\annihec)$. But both subcategories are by definition isomorphism-closed, so they must agree.
\end{proof}

\begin{lemma}\label{anni to uni}
    $\bdermod{\unihec}=\langle\bdermod{\annihec}\rangle_{\bdermod{\unihec}}$
\end{lemma}

\begin{proof}
    Inclusion of the right side in the left is immediate as all $\annihec$-modules are $\unihec$-modules.
    
    Let $\cx{M}$ be an object in $\bdermod{\unihec}$. Then by \cref{Vigneras background}, have some finite $N$ such that $\annihilator^N\cx{M}=0$. Now by \cref{everything is noetherian}, $\annihilator$ is finitely generated, and $\cx{M}$ can be taken to have finitely generated entries by definition of $\bdermod{\unihec}$, and hence the complexes $\annihilator^i\cx{M}$ also have finitely generated entries. Hence the quotients $\annihilator^i\cx{M}/\annihilator^{i+1}\cx{M}$ are objects in $\bdermod{\annihec}$. Hence $\cx{M}$ is a repeated extension of complexes in $\bdermod{\annihec}$, and so is in $\langle\bdermod{\annihec}\rangle$.
\end{proof}

\begin{theorem}\label{vigneras generator derived generates}
    $\bdermod{\unihec}=\langle \annigen\rangle_{\bdermod{\unihec}}$.
\end{theorem}

\begin{proof}
    By \cref{Vigneras generator}, $\annigen$ is a progenerator for $\anniblock$. Hence \\ $\per(\annihec)=\langle \annigen\rangle_{\per(\annihec)}$. Thus we are done by \cref{mod is proj} and \cref{anni to uni}.
\end{proof}

\begin{lemma}\label{induced finite generators are derived equivalent}
    $\langle \phorind_{\fingp,\maxcomp}^{\padicgp}(\fannig)\rangle_{\bdermod{\unihec}}=\langle \ourgen\rangle_{\bdermod{\unihec}}$.
\end{lemma}

\begin{proof}
    By \cref{our gen is in the unipotent}, both $\ourgen$ and $\phorind_{\fingp,\maxcomp}^{\padicgp}(\fannig)$ are objects of $\bdermod{\unihec}$.

    By \cref{finite generators are derived equivalent}, there is a short exact sequence, two entries of which are the direct summands of $\fourg$, and two of which are the direct summands of $\fannig$. But as parahoric induction is exact, we thus get an analogous statement for $\phorind_{\fingp,\maxcomp}^{\padicgp}(\fannig)$ and $\ourgen=\phorind_{\fingp,\maxcomp}^{\padicgp}(\fourg)$. Thus they each classically generate the other.
\end{proof}

\begin{theorem}
    $\bdermod{\unihec}=\langle \ourgen\rangle_{\bdermod{\unihec}}$
\end{theorem}

\begin{proof}
    We already know $\ourgen\in\bdermod{\unihec}$ from \cref{our gen is in the unipotent}, and from \cref{vigneras generator derived generates} that $\bdermod{\unihec}=\langle \annigen\rangle_{\bdermod{\unihec}}$, so it suffices to prove that $\annigen\in\langle \ourgen\rangle_{\bdermod{\unihec}}$. But we also know from \cref{induced finite generators are derived equivalent} that $\langle \phorind_{\fingp,\maxcomp}^{\padicgp}(\fannig)\rangle_{\bdermod{\unihec}}=\langle \ourgen\rangle_{\bdermod{\unihec}}$. Hence we only need to show $\annigen\in\langle\phorind_{\fingp,\maxcomp}^{\padicgp}(\fannig)\rangle_{\bdermod{\unihec}}$. But by \cref{annigen induced from fingen} we have that $\annigen\cong\phorind_{\fingp,\maxcomp}^{\padicgp}(\fannig)$.
\end{proof}

\begin{corollary}
    Let $\gencx$ be a projective resolution of $\ourgen$ in $\smrep$. There is a triangulated equivalence $\bdermod{\unihec}\simeq\per(\dgend(\gencx))$.
\end{corollary}

\begin{proof}
    This follows applying \cref{triangulated equivalence} to the category $\bdermod{\unihec}$ and its classical generator $\ourgen$.
\end{proof}

\printbibliography

\end{document}